\documentclass{article}
\usepackage[utf8]{inputenc}
\usepackage[normalem]{ulem}
\usepackage[letterpaper]{geometry}
\usepackage[colorlinks=true,citecolor=cyan,linkcolor=blue,urlcolor=blue]{hyperref}
\usepackage{kmath}
\usepackage[shortlabels]{enumitem}
\usepackage{graphicx}
\usepackage{bigints}
\usepackage{shuffle}
\usepackage[makeroom]{cancel}
\usepackage{wrapfig}
\usetikzlibrary{decorations.markings}
\Crefname{equation}{}{}

\title{\vspace{-0.7in}Deformation quantization generates all multiple zeta values}
\date{September 27, 2024} 
\author{Kelvin Ritland\footnote{McGill University, \href{mailto:kelvin.ritland@mail.mcgill.ca}{kelvin.ritland@mail.mcgill.ca}}}
\newcommand{\conjspace}[1]{\overline{#1}}
\DeclareRobustCommand{\lhp}{\conjspace\uhp}

\DeclareMathOperator{\sgn}{sgn}
\DeclareMathOperator*{\Res}{Res}
\DeclareMathOperator*{\NRes}{NRes}

\DeclareRobustCommand{\conj}{\overline}
\DeclareRobustCommand{\closure}{\underline}
\DeclareRobustCommand{\hyperlog}[3]{L\left(#1\middle|#3\middle|#2\right)} 

\DeclareRobustCommand{\disk}{\mathfrak{D}}
\DeclareRobustCommand{\conf}{\mathfrak{C}}
\DeclareRobustCommand{\rsphere}{\mathfrak{Z}}
\DeclareRobustCommand{\univreals}{\mathfrak{R}}
\DeclareRobustCommand{\Conf}{\mathsf{Conf}}

\DeclareRobustCommand{\nauts}{\caln}
\DeclareRobustCommand{\mzvs}{\calz}
\DeclareRobustCommand{\mzvswithhalf}{\widetilde\calz}
\DeclareRobustCommand{\arnolds}{\cala}

\DeclareRobustCommand{\kontgraphs}{\calk}
\DeclareRobustCommand{\polywords}{\calw}
\DeclareRobustCommand{\allstrings}{\cals}
\DeclareRobustCommand{\nautstrings}{\cali}
\DeclareRobustCommand{\syntaxtrees}{\calt}
\DeclareRobustCommand{\appendring}{\calf}

\DeclareRobustCommand{\prepend}{\mu}
\DeclareRobustCommand{\append}{\nu}
\DeclareRobustCommand{\graphform}{\alpha}
\DeclareRobustCommand{\graphof}{G}
\DeclareRobustCommand{\stringof}{I}
\DeclareRobustCommand{\coeffof}{c}

\begin{document}
\maketitle

\begin{abstract}
   Banks--Panzer--Pym have shown that the volume integrals appearing in Kontsevich's deformation quantization formula always evaluate to integer-linear combinations of multiple zeta values (MZVs).
   We prove a sort of converse, which they conjectured in their work, namely that with the logarithmic propagator: (1) the coefficients associated to the graphs appearing at order $\hbar^n$ in the quantization formula generate the $\qs$-vector space of weight-$n$ MZVs, and (2) the set of all coefficients generates the $\ints$-module of MZVs.
   In order to prove this result, we develop a new technique for integrating Kontsevich graphs using polylogarithms and apply it to an infinite subset of Kontsevich graphs.
   Then, using the binary string representation of MZVs and the Lyndon word decomposition of binary strings,
   we show that this subset of graphs generates all MZVs.
\end{abstract}

\section{Overview}

In a landmark result, Kontsevich showed in 1997 that the process of \emph{deformation quantization} is possible for any Poisson manifold \cite{kontsevich_2003}.
The quantization of a Poisson manifold $M$ is a deformation of the associative multiplication on $C^\infty(M)$ to a \emph{star product} given by a formal power series in a deformation parameter $\hbar$.  The terms in this series are indexed by certain directed graphs $\Gamma$, which we call \emph{Kontsevich graphs} in this paper.

The coefficient assigned to each graph $\Gamma$ has the form $c_\Gamma = \int_{\conf_{n,m}}\graphform_\Gamma$ where $\graphform_\Gamma$ is a volume form determined by the graph and $\conf_{n,m}$ is a moduli space of marked holomorphic disks.
To construct the volume form $\graphform_\Gamma$, we associate a differential form known as the \emph{propagator} to each edge of the graph and take the wedge product of the forms over all edges. Different choices of propagator are possible.  The one that 
Kontsevich originally specified is nowadays called the \emph{harmonic propagator}, and later work by Alekseev--Rossi--Torossian--Willwacher \cite{alekseev2016logarithms} rigorously established the \emph{logarithmic propagator} (first stated by Kontsevich in \cite{kont_operadsandmotives}).
In addition, Rossi--Willwacher described an infinite family of propagators interpolating between the harmonic and logarithmic propagator \cite{rossi2014interpolating}.

Subsequent work by Banks--Panzer--Pym constructed an explicit algorithm for computing $c_\Gamma$ for any choice of these propagators \cite{Banks_2020},
and showed that $c_\Gamma$ is, up to a normalization factor, an integer-linear combination of
\emph{multiple zeta values} (MZVs)
\begin{align*}
    \zeta(n_1,\dots,n_d) = \sum_{0<k_1<\dots<k_d} \frac{1}{k_1^{n_1}\cdots k_d^{n_d}} \in \reals.
\end{align*}
where $n_1,\dots,n_d$ are positive integers with $n_d \geq 2$.
The sum $\sum_{i}n_i$ is the \emph{weight} of the MZV and $d$ is the $\emph{depth}$ of the MZV.
Let $\mzvs$ be the subring of $\comps$ generated by $\emph{normalized MZVs}$:
\begin{align*}
    \mzvs = \left\langle\frac{\zeta(n_1,\dots,n_d)}{(2\pi i)^{n_1+\dots+n_d}} \;\middle|\; n_i \geq 1, n_d \geq 2 \right\rangle  \subset \comps.
\end{align*}
Let $\mzvs^0 = \ints$ and $\mzvs^1 = \{0\}$, and for other $n\in\nats$ let $\mzvs^n\subset\mzvs$ be the $\ints$-module of weight-$n$ normalized MZVs.
Define a filtered subring $\mzvswithhalf$ of $\comps$ by
\begin{align*}
    \mzvswithhalf^n &= \mzvs^n + \frac{1}{2}\mzvs^{n-1} &
    \mzvswithhalf &= \bigunion_{n\geq 0}\mzvs^n.
\end{align*}
Then, more precisely, Banks--Panzer--Pym proved that a graph $\Gamma$ that appears at order $\hbar^n$ in the star product has $c_\Gamma \in \mzvswithhalf^n$. They also conjectured a converse:
\begin{conjecture*}[\cite{Banks_2020} Conjecture 1.4]
The integrals appearing at order $\hbar^n$ in the logarithmic star product generate $\mzvswithhalf^n$ as a $\ints$-module.
\end{conjecture*}
The main result of this paper is that the conjecture is true if we use coefficients in $\qs$ or drop the weight restriction:
\begin{theorem}
\label{thm:intro_generate_mzvswithhalf_tensor_q}
    The integrals at order $\hbar^n$ in the logarithmic star product generate $\mzvswithhalf^n\tensor\qs$ as a $\qs$-vector space.
\end{theorem}
\begin{theorem}
\label{thm:intro_generate_mzvswithhalf}
    The integrals at all orders in the logarithmic star product generate $\mzvswithhalf$ as a $\ints$-module.
\end{theorem}

The proof is constructive and yields an effective algorithm for writing any MZV as a $\qs$-linear combination of graphs.
The proof also develops a new method to integrate various Kontsevich graphs, which may be of independent interest.

\subsection{Overview of the proof}

We consider the class of Kontsevich graphs obtained by starting from the base graph
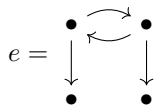
\[
e = 
\begin{tikzpicture}[baseline=(current bounding box.center)]
\begin{scope}[decoration={
    markings,
    mark=at position 1 with {\arrow{>}}}
    ] 
    \node (A) at (0,0) {$\bullet$};
    \node (B) at (1,0) {$\bullet$};
    \node (C) at (0,1) {$\bullet$};
    \node (D) at (1,1) {$\bullet$};
    \draw[postaction=decorate] (C) to[bend left] (D);
    \draw[postaction=decorate] (D) to[bend left] (C);
    \draw[postaction=decorate] (C) to (A);
    \draw[postaction=decorate] (D) to (B);
\end{scope}
\end{tikzpicture}
\]
and applying some sequence of the operations
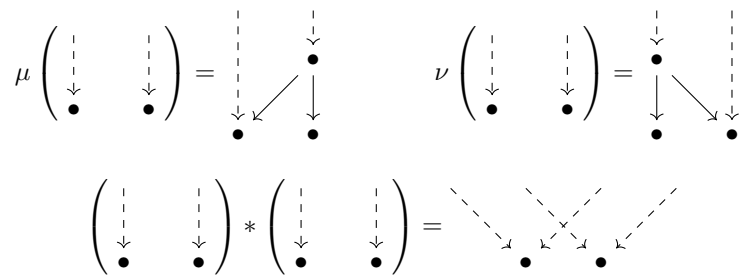
\[
\prepend\left(
\begin{tikzpicture}[baseline=(current bounding box.center)]
\begin{scope}[decoration={
    markings,
    mark=at position 1 with {\arrow{>}}}
    ] 
    \node (A) at (0,0) {$\bullet$};
    \node (B) at (1,0) {$\bullet$};
    \draw[postaction=decorate,dashed] (0,1) to (A);
    \draw[postaction=decorate,dashed] (1,1) to (B);
\end{scope}
\end{tikzpicture}
\right) = 
\begin{tikzpicture}[baseline=(current bounding box.center)]
\begin{scope}[decoration={
    markings,
    mark=at position 1 with {\arrow{>}}}
    ] 
    \node (A) at (0,-1) {$\bullet$};
    \node (B) at (1,0) {$\bullet$};
    \node (C) at (1,-1) {$\bullet$};
    \draw[postaction=decorate,dashed] (0,0.66) to (A);
    \draw[postaction=decorate,dashed] (1,0.66) to (B);
    \draw[postaction=decorate] (B) to (A);
    \draw[postaction=decorate] (B) to (C);
\end{scope}
\end{tikzpicture}
\quad\quad\quad\quad
\append\left(
\begin{tikzpicture}[baseline=(current bounding box.center)]
\begin{scope}[decoration={
    markings,
    mark=at position 1 with {\arrow{>}}}
    ] 
    \node (A) at (0,0) {$\bullet$};
    \node (B) at (1,0) {$\bullet$};
    \draw[postaction=decorate,dashed] (0,1) to (A);
    \draw[postaction=decorate,dashed] (1,1) to (B);
\end{scope}
\end{tikzpicture}
\right) = 
\begin{tikzpicture}[baseline=(current bounding box.center)]
\begin{scope}[decoration={
    markings,
    mark=at position 1 with {\arrow{>}}}
    ] 
    \node (A) at (0,0) {$\bullet$};
    \node (B) at (1,-1) {$\bullet$};
    \node (C) at (0,-1) {$\bullet$};
    \draw[postaction=decorate,dashed] (0,0.66) to (A);
    \draw[postaction=decorate,dashed] (1,0.66) to (B);
    \draw[postaction=decorate] (A) to (B);
    \draw[postaction=decorate] (A) to (C);
\end{scope}
\end{tikzpicture}
\]
\[
\left(
\begin{tikzpicture}[baseline=(current bounding box.center)]
\begin{scope}[decoration={
    markings,
    mark=at position 1 with {\arrow{>}}}
    ] 
    \node (A) at (0,0) {$\bullet$};
    \node (B) at (1,0) {$\bullet$};
    \draw[postaction=decorate,dashed] (0,1) to (A);
    \draw[postaction=decorate,dashed] (1,1) to (B);
\end{scope}
\end{tikzpicture}
\right)
*
\left(
\begin{tikzpicture}[baseline=(current bounding box.center)]
\begin{scope}[decoration={
    markings,
    mark=at position 1 with {\arrow{>}}}
    ] 
    \node (A) at (0,0) {$\bullet$};
    \node (B) at (1,0) {$\bullet$};
    \draw[postaction=decorate,dashed] (0,1) to (A);
    \draw[postaction=decorate,dashed] (1,1) to (B);
\end{scope}
\end{tikzpicture}
\right)
=
\begin{tikzpicture}[baseline=(current bounding box.center)]
\begin{scope}[decoration={
    markings,
    mark=at position 1 with {\arrow{>}}}
    ] 
    \node (A) at (0,0) {$\bullet$};
    \node (B) at (1,0) {$\bullet$};
    \draw[postaction=decorate,dashed] (-1,1) to (A);
    \draw[postaction=decorate,dashed] (0,1) to (B);
    \draw[postaction=decorate,dashed] (1,1) to (A);
    \draw[postaction=decorate,dashed] (2,1) to (B);
\end{scope}
\end{tikzpicture}
\]
where the bottom nodes indicate the external nodes of Kontsevich graphs.
For example, the graph $\prepend(e) * e$ is:
\[
\prepend(e) * e = \prepend\left(
\begin{tikzpicture}[baseline=(current bounding box.center)]
\begin{scope}[decoration={
    markings,
    mark=at position 1 with {\arrow{>}}}
    ] 
    \node (A) at (0,0) {$\bullet$};
    \node (B) at (1,0) {$\bullet$};
    \node (C) at (0,1) {$\bullet$};
    \node (D) at (1,1) {$\bullet$};
    \draw[postaction=decorate] (C) to[bend left] (D);
    \draw[postaction=decorate] (D) to[bend left] (C);
    \draw[postaction=decorate] (C) to (A);
    \draw[postaction=decorate] (D) to (B);
\end{scope}
\end{tikzpicture}
\right)
*
\left(
\begin{tikzpicture}[baseline=(current bounding box.center)]
\begin{scope}[decoration={
    markings,
    mark=at position 1 with {\arrow{>}}}
    ] 
    \node (A) at (0,0) {$\bullet$};
    \node (B) at (1,0) {$\bullet$};
    \node (C) at (0,1) {$\bullet$};
    \node (D) at (1,1) {$\bullet$};
    \draw[postaction=decorate] (C) to[bend left] (D);
    \draw[postaction=decorate] (D) to[bend left] (C);
    \draw[postaction=decorate] (C) to (A);
    \draw[postaction=decorate] (D) to (B);
\end{scope}
\end{tikzpicture}
\right)
 = 
 \left(
\begin{tikzpicture}[baseline=(current bounding box.center)]
\begin{scope}[decoration={
    markings,
    mark=at position 1 with {\arrow{>}}}
    ] 
    \node (A) at (0,0) {$\bullet$};
    \node (B) at (1,0) {$\bullet$};
    \node (C) at (1,1) {$\bullet$};
    \node (D) at (0,2) {$\bullet$};
    \node (E) at (1,2) {$\bullet$};
    \draw[postaction=decorate] (D) to[bend left] (E);
    \draw[postaction=decorate] (E) to[bend left] (D);
    \draw[postaction=decorate] (E) to (C);
    \draw[postaction=decorate] (D) to (A);
    \draw[postaction=decorate] (C) to (A);
    \draw[postaction=decorate] (C) to (B);
\end{scope}
\end{tikzpicture}
\right)
*
\left(
\begin{tikzpicture}[baseline=(current bounding box.center)]
\begin{scope}[decoration={
    markings,
    mark=at position 1 with {\arrow{>}}}
    ] 
    \node (A) at (0,0) {$\bullet$};
    \node (B) at (1,0) {$\bullet$};
    \node (C) at (0,1) {$\bullet$};
    \node (D) at (1,1) {$\bullet$};
    \draw[postaction=decorate] (C) to[bend left] (D);
    \draw[postaction=decorate] (D) to[bend left] (C);
    \draw[postaction=decorate] (C) to (A);
    \draw[postaction=decorate] (D) to (B);
\end{scope}
\end{tikzpicture}
\right)
=
\begin{tikzpicture}[baseline=(current bounding box.center)]
\begin{scope}[decoration={
    markings,
    mark=at position 1 with {\arrow{>}}}
    ] 
    \node (A) at (1,0) {$\bullet$};
    \node (B) at (2,0) {$\bullet$};
    \node (C) at (1,1) {$\bullet$};
    \node (D) at (0,2) {$\bullet$};
    \node (E) at (1,2) {$\bullet$};
    \draw[postaction=decorate] (D) to[bend left] (E);
    \draw[postaction=decorate] (E) to[bend left] (D);
    \draw[postaction=decorate] (E) to (C);
    \draw[postaction=decorate] (D) to[bend right] (A);
    \draw[postaction=decorate] (C) to (A);
    \draw[postaction=decorate] (C) to (B);

    \node (H) at (2,1) {$\bullet$};
    \node (I) at (3,1) {$\bullet$};
    \draw[postaction=decorate] (H) to[bend left] (I);
    \draw[postaction=decorate] (I) to[bend left] (H);
    \draw[postaction=decorate] (H) to (A);
    \draw[postaction=decorate] (I) to (B);
\end{scope}
\end{tikzpicture}
\]

To integrate these graphs, we use the binary representation of normalized MZVs due to Le-Murakami \cite{mzvrepn1995} and Kontsevich.
For the sequence of integers $(n_1,\dots,n_d)$ defining $\zeta(n_1,\dots,n_d)$ we have that
\begin{align*}
    \frac{(-1)^d\zeta(n_1,\dots,n_d)}{(2\pi i)^{n_1 + \dots + n_d}} &= \int_0^1 \omega_0^{n_d-1}\omega_1\omega_0^{n_{d-1}-1}\omega_1\cdots \omega_1\omega_0^{n_1-1}\omega_1
\end{align*}
where this integral is Chen's iterated integral \cite{chen1977iterated} and $\omega_a(x) = \frac{1}{2\pi i}\frac{\dd x}{x - a}$. 
By considering the sequence of forms $\omega_0$ and $\omega_1$ that appear in the integrand we have a mapping from binary strings to normalized MZVs:
\begin{align*}
    0^{n_d-1}10^{n_{d-1}-1}1\cdots 10^{n_1-1}1 \mapsto (-1)^d\frac{\zeta(n_1,\dots,n_d)}{(2\pi i)^{n_1+\dots+n_d}}.
\end{align*}
For instance, 
\begin{align*}
    011 \mapsto \frac{\zeta(1,2)}{(2\pi i)^{3}} &= \int_0^1 \omega_0\omega_1\omega_1 = \frac{1}{(2\pi i)^3}\int_0^1 \frac{\dd x_1}{x_1} \int_0^{x_1} \frac{\dd x_2}{x_2-1} \int_0^{x_2} \frac{\dd x_3}{x_3-1}.
\end{align*}

Let $\stringof$ be the map that associates to one of the graphs above the formal linear combination of binary strings defined recursively by the following rules:
\begin{align*}
    \stringof(e) &= 01 \\
    \stringof(\prepend(\Gamma)) &= 0\cdot\stringof(\Gamma) \\
    \stringof(\append(\Gamma)) &= \stringof(\Gamma)\cdot 1 \\
    \stringof(\Gamma_1 * \Gamma_2) &= \stringof(\Gamma_1) \shuffle \stringof(\Gamma_2)
\end{align*}
where $\shuffle$ is the commutative shuffle product of binary strings and $\cdot$ is string concatenation.
We prove that for a graph $\Gamma$, the integral $c_\Gamma$ is equal to the linear combination of MZVs determined by $I(\Gamma)$, i.e., $I$ computes the weight of a graph in binary form.

It is difficult to prove this result using past work.
In particular, the general algorithm described by Banks--Panzer--Pym is highly computationally intensive: graphs with more than five nodes tend to result in thousands of terms to track when implemented on a computer, making it infeasible for proving a result like this. 
Our strategy builds on this approach but with some key simplifications to make it tractable to carry out by hand.

The basic objects used by Banks--Panzer--Pym are \emph{polylogarithms},
which are functions on the integration domain defined by
the iterated integral of certain differential forms over a smooth path $\gamma$.
They depend on the homotopy class of $\gamma$, which means they usually posses some monodromy.
Their algorithm then uses polylogarithms to carry out the integration one variable at a time.
In particular, given $\graphform_\Gamma$ depending on some variable of integration $x$,
they take a na\"ive (usually multivalued) primitive of $\graphform_\Gamma$ with respect to $x$ in terms of polylogarithms.
Then, they `correct' the primitive with a closed form that cancels out the monodromy to obtain a single-valued primitive, and then they apply Stokes' theorem to integrate out $x$.
This single integration step stays within the space of polylogarithms, and so it can be repeated until all variables are integrated out.

We improve on this algorithm through the identification of a class of single valued polylogarithms, which we call $\emph{nautical}$ polylogarithms.
We show that, for the class of graphs in this paper, we can always take a nautical primitive when integrating fibrewise,
and we thus avoid having to deal with the complexities of cancelling out monodromy present in the general algorithm.
The relative simplicity of the primitives then allow us to show that $I$ indeed computes $c_\Gamma$.
Outside of the class of graphs in this paper,
we have also been able to use these polylogarithms to integrate the Bernoulli graphs described in \cite{didier_bernoulligraphs,vinay_bernoulligraphs},
the wheel graphs described in \cite{merkulov10_wheelgraphs}, and arbitrary sized ladders (\Cref{ex:hgraph_single_rung}) by hand fairly easily.

Then, using the formula for $c_\Gamma$ given by $I$ we show our main result, that all MZVs are generated, as follows. 
First, note that the formula implies that the space of all possible binary strings we can obtain (and hence MZVs we may obtain) is closed under the operations $\prepend(s) := 0\cdot s$, $\append(s) := s \cdot 1$, and $\shuffle$.
We then show that any binary string corresponding to an MZV may be written as a $\qs$-linear combination of operations of $\prepend,\append$, and $\shuffle$ starting from the string 01,
which implies \Cref{thm:intro_generate_mzvswithhalf_tensor_q}.
For this proof, we use the Lyndon word polynomial basis for the shuffle algebra of binary strings to decompose a string into a $\qs$-linear combination of shuffle products of Lyndon words (here is where $\qs$ coefficients are necessary),
and then reduce the length of those Lyndon words inductively using the operations $\prepend$ and $\append$.
\begin{ex}
For the normalized MZV $\frac{-1}{(2\pi i)^5}\zeta(1,2,2)$ with binary representation $01011$ our binary string decomposition is:
\begin{align*}
    01011 = \append(0101) &= \append\left(\frac{1}{2}(01\shuffle 01) - 2(0011)\right) 
    = \frac{1}{2}\append(01 \shuffle 01) - 2\append(\append(\prepend(01)))
\end{align*}
so we have that $c_\Gamma = \frac{-1}{(2\pi i)^5}\zeta(1,2,2)$ for the linear combination of graphs
\[
\Gamma = \frac{1}{2}\append(e * e) - 2\append(\append(\prepend(e))) = 
\frac{1}{2}\left(
\begin{tikzpicture}[baseline=(current bounding box.center)]
\begin{scope}[decoration={
    markings,
    mark=at position 1 with {\arrow{>}}},
    scale=1.1
    ] 
    \node (A) at (0,0) {$\bullet$};
    \node (B) at (1,0) {$\bullet$};
    \node (C) at (0,1) {$\bullet$};
    \node (D) at (-1,2) {$\bullet$};
    \node (E) at (0,2) {$\bullet$};
    \node (F) at (1,2) {$\bullet$};
    \node (G) at (2,2) {$\bullet$};
    \draw[postaction=decorate] (D) to[bend left] (E);
    \draw[postaction=decorate] (E) to[bend left] (D);
    \draw[postaction=decorate] (F) to[bend left] (G);
    \draw[postaction=decorate] (G) to[bend left] (F);
    \draw[postaction=decorate] (D) to (C);
    \draw[postaction=decorate] (E) to (B);
    \draw[postaction=decorate] (F) to (C);
    \draw[postaction=decorate] (G) to (B);
    \draw[postaction=decorate] (C) to (A);
    \draw[postaction=decorate] (C) to (B);
\end{scope}
\end{tikzpicture}
\right)
- 2\left(
\begin{tikzpicture}[baseline=(current bounding box.center)]
\begin{scope}[decoration={
    markings,
    mark=at position 1 with {\arrow{>}}},
    scale=0.66*1.1
    ] 
    \node (A) at (0,0) {$\bullet$};
    \node (B) at (1,0) {$\bullet$};
    \node (C) at (0,1) {$\bullet$};
    \node (D) at (0,2) {$\bullet$};
    \node (E) at (1,2) {$\bullet$};
    \node (F) at (0,3) {$\bullet$};
    \node (G) at (1,3) {$\bullet$};
    \draw[postaction=decorate] (F) to[bend left] (G);
    \draw[postaction=decorate] (G) to[bend left] (F);
    \draw[postaction=decorate] (F) to (D);
    \draw[postaction=decorate] (G) to (E);
    \draw[postaction=decorate] (E) to (D);
    \draw[postaction=decorate] (E) to (B);
    \draw[postaction=decorate] (D) to (B);
    \draw[postaction=decorate] (D) to (C);
    \draw[postaction=decorate] (C) to (A);
    \draw[postaction=decorate] (C) to (B);
\end{scope}
\end{tikzpicture}
\right),
\]
as desired.
\end{ex}
\Cref{thm:intro_generate_mzvswithhalf} then follows by observing that $\qs$ is generated by $\ints$-linear combinations of all graphs, so the $\qs$-span and $\ints$-span of all graphs are the same.

In the last section of the paper we show (assuming the conjectured basis $\{\zeta(2,3), \zeta(3,2)\}$ for $\mzvs^5\tensor \qs$ from \cite{hoffman97}) that our class of graphs of weight 5 do not span $\mzvswithhalf^5$ as a $\ints$-module.
It would be interesting to investigate whether our result can be adapted to show that the coefficients appearing at weight $n$ span $\mzvswithhalf^n$ as an $\ints$-module, not just as a $\qs$-vector space.

\subsection*{Acknowledgements}
We thank Erik Panzer for his initial code and suggestions, and Brent Pym for his valuable guidance through this project. 
We acknowledge the support of the Natural Sciences and Engineering Research Council of Canada through a Canada Graduate Scholarship -- Master's (CGS-M), and through Discovery Grant RGPIN-2020-05191 of Pym.

\section{Polylogarithms and Kontsevich weights}

In this section, we recall some preliminaries following the notation established in Section 2 of \cite{Banks_2020}.
Most of the content in this section is not original to this paper; the main exception is \Cref{def:integration_notation},
where we introduce a polylogarithm in two variables depending on a Kontsevich graph $\Gamma$,
generalizing the usual integral defining the coefficient $c_\Gamma$.

\subsection{Coordinates and spaces}
Let $\uhp$ and $\lhp$ be the open upper and lower half planes of $\comps$, the bar here standing for complex conjugation.

Let $U$ and $R$ be finite sets of \emph{labels} with $R$ totally ordered and $2|U|+|R| \geq 2$.
Let $R_\infty = R\disunion\{\infty\}$ be $R$ with an additional label called $\infty$,
and $\conj U = \set{\conj u}{u \in U}$ be the set of \emph{conjugate labels} of $U$.
Let $S = U \disunion \conj U \disunion R$ and let $S_\infty = U \disunion \conj U \disunion R_\infty$. 

A \emph{configuration of $U$} is an embedding $U \hookrightarrow \uhp$ and a \emph{configuration of $R$} is an embedding $R \hookrightarrow \reals$ that respects the total ordering on $R$.
Explicitly, we can write all such configurations as
\begin{align*}
\Conf_U(\uhp) &= \set{ (z_u)_{u\in U}}{z_{u_1} \neq z_{u_2} \text{ if } u_1\neq u_2} \subset \uhp^U \\
\Conf_{R,+}(\reals) &= \set{ (z_r)_{r\in R}}{z_{r_1} < z_{r_2} \text{ if } r_1 < r_2} \subset \reals^R .
\end{align*}
For a given configuration $C_U \in \Conf_U(\uhp)$ we also consider $\conj U$ to be embedded in $\lhp$ via complex conjugation:
\begin{align*}
    \conj C_U := (\conj z_u)_{z_u \in C_U} \in \lhp^U
\end{align*}
The space of all such embeddings is isomorphic to $\Conf_U(\uhp)$.

A configuration $C_U \in \Conf_U(\uhp)$ and $C_R \in \Conf_{R,+}(\reals)$ induces an embedding of $S \hookrightarrow \comps$, given by $C_S = C_U \disunion C_R \disunion \conj C_U$.
These points in $\comps$ will be called \emph{marked points}.
Intuitively, each configuration $C_S$ consists of $|U|$ marked points in $\uhp$ labelled by $U$, together with $|R|$ marked points on $\reals$ labelled by $R$, and $|\conj U|$ marked points in $\lhp$ labelled by $\conj U$ so that the whole collection is preserved by complex conjugation.

Two configurations of $R$ and $U$ are \emph{equivalent} if they differ by a conformal transformation of $\uhp$ preserving $\infty$. 
\begin{define}
The \emph{moduli space of $U$ marked points in the interior and $R$ marked points on the boundary}
is the set of equivalence classes of configurations of labels.
It is explicitly given by
\begin{align*}
    \conf_{U,R} \cong (\Conf_U(\uhp) \cross \Conf_{R,+}(\reals)) / G
\end{align*}    
where $G$ is the group of conformal transformations of $\uhp$ that preserve $\infty$.
\end{define}
$\conf_{U,R}$ is a real analytic manifold of dimension $2|U| + |R| - 2$.

From here on we do not distinguish the labels in $U, \conj U, R$ from the corresponding marked points in $\comps$,
i.e., we will use the labels $u\in U, r\in R$, and $\conj u\in\conj U$ to indicate the marked points $z_u\in\uhp$, $z_r\in\reals$, and $\conj z_u \in \lhp$.
When the exact configuration is unspecified we also consider $u,r,\conj u$ as variables or coordinates on $\conf_{U,R}$, 
taking values in in $\uhp$, $\reals$, or $\lhp$, respectively.

We would also like to consider functions where a variable may be either in $U$ or $R$.
To define these functions uniformly, 
we define a space where some points are allowed to be in the interior or the boundary:
\begin{define}\label{def:closure_conf_space}
For a finite set of labels $T$ and subset $X \subset T$, we denote by 
\begin{align*}
    \closure\conf_{T}^X = \bigdisunion_{(U,R) \in \mathscr{T}}\conf_{U,R} 
\end{align*}
where the set $\mathscr{T}$ is the set of all possible partitions $T = U \disunion R$ with $R$ totally ordered, $2|U|+|R|\geq 2$, and $X \subseteq U$.
  We refer to each $\conf_{U,R}$ in the disjoint union as the $\emph{strata}$ of the space.  When $X = \emptyset$, we write
\[
    \closure\conf_{T} = \closure\conf_{T}^\emptyset.
    \]
\end{define}
Note that the topology of $\closure \conf_T^X$ is that of the disjoint union of the components $\conf_{U,R}$.
This set could alternatively be equipped with a topology as a manifold with corners, whose interior is the set $\conf_{T,\emptyset}$ of configurations where all points are in the upper half-plane, and whose boundary strata are the remaining components $\conf_{U,R}$, corresponding to configurations where the points of $T\setminus X$ collide with $\reals$ but not with any other points.  This second topology is useful for intuition but will not be used in this paper; rather all arguments will be performed on a stratum-by-stratum basis so that the disjoint union topology suffices. This approach allows us to avoid the subtle problem of checking that the polylogarithmic functions we are dealing with are continuous as we pass from one stratum to another.

\subsubsection{The universal disk}
$\conf_{U,R}$ carries a universal disk $\disk_{U,R} \to \conf_{U,R}$ whose fibre at a point $[C_{U,R}] \in \conf_{U,R}$ is isomorphic to the upper half plane punctured at the marked points $C_{U}$.
Concretely, $\disk_{U,R} = \conf_{U\disunion\{x\},R}$ where $x$ is an additional symbol,
with location of the corresponding marked point serving as a coordinate on each fibre.
Each fiber has a natural compactification given by the real oriented blowup of the closed disk at the marked points $U\disunion R$, as shown in \Cref{fig:disk_example}.
\begin{figure}
    \centering
    
\begin{tikzpicture}
\begin{scope}[decoration={
    markings,
    mark=at position 0.5 with {\arrow{>}}}
    ] 
\def\holerad{0.3}
\begin{scope}[shift={(0,0)}]
    \def\radius{1.5}
    \def\perimtheta{(4*asin(\holerad/(2*\radius)))} 
    \def\perimarc{(180 - \perimtheta/2)} 

    \def\holestart{(\holepos+\perimarc/2-180)}
    \def\holeend{(180-\perimarc/2+\holepos)}
    \draw[fill=lightgray] (0,0) circle (\radius);
    \foreach \holepos/\holelabel in {90/$\infty$,290/$r_2$,170/$r_1$,10/$r_3$}{
        \draw[fill=white,draw=none] ({\holepos:\radius}) circle (\holerad);
        \draw[fill=white, postaction={decorate}] ([shift=({\holestart}:\holerad)] {\holepos:\radius}) arc (\holestart:\holeend-360:\holerad);
        \node at ({\holepos:\radius}) {\holelabel};
    };
    
    \foreach \x/\y/\holelabel in {{0.2}/{0.4}/$u_1$,{-0.4}/{-0.4}/$u_2$}{
        \draw[fill=white, postaction={decorate}] ([shift=(270:\holerad)] \x,\y) arc (270:270-360:\holerad);
        \node at ({\x,\y}) {\holelabel};
    }

    \node at (0.5,-0.5) {$x^\cdot$};
\end{scope}
\end{scope}
\end{tikzpicture}
    \caption{The compactification of a fiber in $\disk_{U,R}$ when $R = \{r_1, r_2, r_3\}$ and $U = \{u_1,u_2\}$}
    \label{fig:disk_example}
\end{figure}
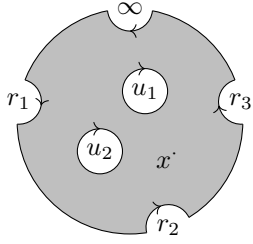

We will also consider the space $\conjspace\disk_{U,R}$, which is defined to be the complex conjugate space of $\disk_{U,R}$. We view the fibres of $\conjspace\disk_{U,R}$ as the lower half plane with punctures at $\conj U$
and with $\conj x$ as a (holomorphic) coordinate.

Just as $\comps\projs^1$ is the union of the upper half plane $\uhp$, the lower half plane $\lhp$ and the real line $\reals$, the universal punctured disk $\disk_{U,R}$ and its conjugate $\conj \disk_{U,R}$ can be glued along a common boundary corresponding to $\reals \setminus R$, to obtain a copy of the Riemann sphere punctured at $S_\infty$. 
This gives a universal family $\rsphere_{U,R} \to \conf_{U,R}$ of punctured genus zero curves equipped with a complex conjugation automorphism.
Let $\univreals_R \subseteq \rsphere_{U,R}$ be the shared boundary corresponding to $\reals \setminus R$.

Finally, we denote by $\closure\disk_T^X \to \closure\conf_T^X$ and $\closure\rsphere_T^X \to \closure\conf_T^X$ the analogous constructions over $\closure\conf_T^X$.
For instance, $\closure\disk_T^X$ is given by the disjoint union of the universal curves $\disk_{U,R}\to \conf_{U,R}$ for each stratum $\conf_{U,R}$ in $\closure\conf_T^X$.

\subsection{Differential forms and polylogarithms}

We denote by $\Omega^1(\conf_{U,R})$ the differential forms on $\conf_{U,R}$ and by $\Omega^1(\rsphere_{U,R})=\Omega^1(\rsphere_{U,R}/\conf_{U,R})$ the relative forms on the universal punctured sphere.
Then we denote by $\Omega^1(\closure\conf_T^X)$ and $\Omega^1(\closure\rsphere_{U,R})$ the direct sum of the respective forms on each stratum of $\closure\conf_T^X$.

Suppose that $x \in U$ and $a \in S = U \disunion \conj U \disunion R$.
Define differential forms $\omega_a(x), \conj\omega_a(x) \in \Omega^1(\rsphere_{U,R})$ by
\begin{align*}
    \omega_a(x) = \frac{1}{2\pi i}\dd\log(x - a) && \text{and} && \conj\omega_a(x) = \frac{1}{2\pi i}\dd\log(\conj x - a),
\end{align*}
where $x$ is the fiber coordinate on $\rsphere_{U,R}$, and let
\begin{align*}
    \arnolds^\bullet(\conf_{U,R}) \subseteq \Omega^\bullet(\conf_{U,R})
\end{align*}
be the subring of the de Rham complex generated by logarithmic derivatives of cross ratios (see \cite[Sec 2.4.2]{Banks_2020} for additional detail).
Note that we are conflating the coordinates indicating the locations of the marked points $z_x, z_{\conj x}, z_a$ with the formal labels $x, \conj x$ and $a \in S$ as discussed above.

We define $\arnolds^\bullet(\closure\conf_T^X)$ as the direct sum of each $\arnolds^\bullet(\conf_{U,R})$ for each stratum of $\closure\conf_T^X$.
Then, for $x \in X$, we have elements $\omega_a(x), \conj\omega_a(x) \in \Omega^1(\closure \conf_T^X)$ whose restriction to each stratum are as above.

\begin{define}
Suppose that $s_1,\dots,s_n \in S$ and let $\gamma : (0,1) \to \rsphere_{U,R}$ be a path lying entirely in a fibre, such that the limits
\begin{align*}
\gamma(0) := \lim_{t \to 0} \gamma(t) && \gamma(1) := \lim_{t \to 1} \gamma(t)
\end{align*}
are well defined and lie in $S_\infty$.  The \emph{iterated integral} in the sense of Chen \cite{chen1977iterated} is defined by
\begin{align*}
    \int_\gamma \omega_{s_1}\cdots\omega_{s_n} := \int_{0 \leq t_n\leq\dots\leq t_1 \leq 1}(\gamma^*\omega_{s_1})(t_1)\dots(\gamma^*\omega_{s_n})(t_n)
\end{align*}
\end{define}

By varying the point in the base $\conf_{U,R}$  and keeping the homotopy class of $\gamma$ in the fibres locally constant, this integral defines a (possibly multivalued) function $\conf_{U,R} \to \comps$ known as a \emph{polylogarithm}.
Regarded as function of just the endpoint $\gamma(1)$ (i.e.~as a function on the fibre) this integral is known as a \emph{hyperlogarithm}.

If we denote by $\gamma_x$ the subpath from $\gamma(0)$ to some point $x$ in the image of $\gamma$ then
$\int_\gamma \omega_{s_1}\cdots\omega_{s_n} = \int_\gamma\omega_{s_1}(x)\int_{\gamma_x} \omega_{s_2}\cdots\omega_{s_n}$, showing the iterated nature of these integrals.

\begin{define}
    Let $\polywords$ be the free ring generated by $S$, i.e. the set of finite $\ints$-linear combinations of words composed of letters of $S$.
    We will use $\cdot$ to denote multiplication (i.e. string concatenation) in $\polywords$ where needed.
\end{define}

As an alternate notation,
for a word $w = s_1\cdots s_n \in \polywords$ and an appropriate path $\gamma$, we will write
\begin{align*}
    \int_\gamma w := \int_\gamma s_1\cdots s_n := \int_\gamma \omega_{s_1}\cdots\omega_{s_n}.
\end{align*}
to indicate the polylogarithm.
Since iterated integrals are homotopy invariant,
when the homotopy class of $\gamma$ relative to its endpoints $\gamma(0)=a$ and $\gamma(1)=b$ is unambiguous 
for $w \in \polywords$, we will write
\begin{align*}
    \hyperlog{a}{b}{w} := \int_a^b w := \int_\gamma w
\end{align*}
for the polylogarithm as well.
\begin{define}
    Suppose that $p \in S$ and let $V \subseteq \comps \setminus S$ be an open set such that $p$ is in the closure of $V$.
    Then a real analytic function $f: V \to \comps$ has \emph{logarithmic singularities at $p$} if
    there exists a holomorphic coordinate $z$ centred
    at $p$ such that
    \begin{align*}
        f(z) = \sum_{j=0}^n f_j \cdot (\log z)^j + \sum_{j=0}^m g_j \cdot (\log \conj z)^j + \sum_{j=0}^k h_j \cdot (\log (z - \conj z))^j
    \end{align*}
    where $f_j, g_j, h_j$ are functions that are real analytic at $p$, and $\log$ is the principal branch of the logarithm.
\end{define}

Note that if $f$ is holomorphic then we need only the first term in this definition (i.e.~we may assume all $g_j$ and $h_j$ are equal to zero) and $f_j$ can be assumed to be holomorphic.
Polylogarithms have at worst logarithmic singularities as marked points collide with each other.

By \cite[Lemma 3.3.16]{panzer_phd}, the integral $\int_\gamma s_1\cdots s_n$ is convergent if and only if $\gamma(1) \neq s_1$ and $\gamma(0) \neq s_n$.
It is possible to treat divergences using regularized limits and tangential base points, but all polylogarithms in this work are convergent so we do not elaborate here.

\begin{define}
    The \emph{sheaf of convergent polylogarithms} on $\conf_{U,R}$ is the subsheaf of rings $\calv_{U,R} \subset C^{\infty}_{\conf_{U,R}}$ generated by polylogarithms $\int_\gamma \omega_{s_1}\dots\omega_{s_n}$
    where 
    (1) $s_1\cdots s_n \in \polywords$ is a word,
    (2) $\gamma$ is a path between marked points $a,b \in S_\infty$ such that $s_1 \neq b$ and $s_n \neq a$, and
    (3) aside from the endpoints $\gamma(0)=a$ and $\gamma(1)=b$, the image of $\gamma$ does not intersect $\{s_1,\dots,s_n\}$.

    The sheaf of convergent polylogarithms on $\closure\conf_T^X$, denoted $\closure\calv_T^X$, is the direct sum of $\calv_{U,R}$ for each stratum $\conf_{U,R}$ in $\closure\conf_T^X$.
\end{define}

Finally, we state some identities satisfied by iterated integrals.

\begin{enumerate}
    \item (Shuffle product) Suppose that $w, v \in \polywords$. Then
    \begin{align}
        \left(\int_\gamma w\right)\left(\int_\gamma v\right) = \int_{\gamma} w \shuffle v \label{eqn:shuffle_prod}
    \end{align}
    where $\shuffle$ is the commutative shuffle product in $\polywords$, i.e., $\int_\gamma : \polywords \to \calv_{U,R}$ is a ring homomorphism.
    
    \item (Path concatenation) Let $\gamma_1 * \gamma_2$ be the concatenation of two paths, first $\gamma_2$ then $\gamma_1$.
    Then
    \begin{align}
        \int_{\gamma_1*\gamma_2} s_1\dots s_n = \sum_{j=0}^k\left(\int_{\gamma_1} s_1\dots s_j\right)\left(\int_{\gamma_2} s_{j+1}\dots s_n\right) \label{eqn:path_concat}
    \end{align}

    \item (Path reversal) We have
    \begin{align}
        \int_{\gamma^{-1}} s_1\cdots s_n = (-1)^n\int_\gamma s_n \cdots s_1 \label{eqn:path_reverse}
    \end{align}
    where $\gamma^{-1}$ is the reversal of $\gamma$.
    Note the reversal of the symbols in the integrand.
\end{enumerate}

\begin{ex}\label{ex:calv_ab_is_mzvs}
    Note that $\conf_{\emptyset,\{A,B\}}$ is a single point.
    The sheaf $\calv_{\emptyset,\{A,B\}}$ is then simply the ring generated by polylogarithms $\int_A^B w$ where $w$ is a string composed of $A$ and $B$ that starts with $A$ and ends with $B$.
    Note the path reversal formula implies that the polylogarithms obtained by integrating over a path from $B$ to $A$ are also generated by this set.
    By identifying $A$ with 0 and $B$ with 1 by an appropriate conformal transformation, we obtain the integrals $\int_0^1 w$ where $w$ is a binary string starting with 0 and ending with 1,
    which is exactly the binary representation of normalized MZVs.
    We thus have $\calv_{\emptyset,\{A,B\}} = \mzvs$.
\end{ex}

\subsection{Kontsevich weight integration}

Here we recall the integrals used to define the weights $c_\Gamma$ with the logarithmic propagator, as described by Kontsevich/Alekseev–Rossi–Torossian–Willwacher \cite{kont_operadsandmotives,kontsevich_2003,alekseev2016logarithms}.

\begin{define}
For $n \ge 0$, a \emph{Kontsevich star product graph of weight $n$} is the data of a directed graph $\Gamma$ with $n+2$ nodes and $2n$ edges, that has
\begin{enumerate}
    \item $n$ nodes with outdegree two, called \emph{internal} nodes,
    \item two nodes with outdegree zero, called \emph{external} nodes,
\end{enumerate}
together with an ordering of its edges from 1 to $2n$.  
\end{define}

We will refer to these graphs simply as Kontsevich graphs in this paper; we do not consider the more general graphs with any number of external nodes used in his formality theorem.
We also will usually label the external nodes with symbols $A$ and $B$.

\begin{define}
    Let $\kontgraphs$ be  the free abelian group generated by isomorphism classes of Kontsevich graphs, modulo the relation $\Gamma = \sgn(\sigma)\Gamma'$ where $\Gamma'$ is any graph that is isomorphic to $\Gamma$ but with the edges reordered by applying a permutation $\sigma$.
    We equip $\kontgraphs$ with the grading $\kontgraphs = \bigoplus_{n =0}^\infty \kontgraphs^n$, where $\kontgraphs^n \subset \kontgraphs$ is the subgroup generated by  graphs of weight $n$.
\end{define}

\begin{define}
For $x,y \in U$, the  \emph{logarithmic propagator} from $x$ to $y$ is the one-form
\begin{align*}
    \alpha_{x\to y} = \frac{1}{2\pi i}\dd\log\left(\frac{x - y}{\conj x - y}\right) \in \Omega^1(\conf_{U,R}).
\end{align*}
\end{define}
Note that we are again considering $x,y$ to be variables,
by conflating the labels $x,y \in U$ with their corresponding marked points.

For a Kontsevich graph $\Gamma$,
and a choice of labeling of internal nodes by $U$ and external nodes by $R$,
let
\begin{align*}
    \graphform_\Gamma := \bigwedge_{i=1}^{2n}\alpha_{e_{i0} \to e_{i1}}
\end{align*}
where $e_{i0}$ and $e_{i1}$ are the endpoints of edge $e_i$ in $\Gamma$. 
Then the integral
\begin{align*}
    c_\Gamma = \int_{\conf_{U,R}}\graphform_\Gamma
\end{align*}
is the weight integral that appears in Kontsevich's star product formula.
It is known to be absolutely convergent \cite{alekseev2016logarithms},
and evaluates to a $\ints$-linear combination of normalized multiple zeta values in $\mzvswithhalf^n$ for a weight $n$ graph \cite{Banks_2020}.
Intuitively,
the integral over $\conf_{U,R}$ can be thought of as an integral over $\uhp$ with respect to all marked points $u\in U$.

Note that since even permutations of the edges leave the wedge product invariant, $c_\Gamma$ depends only on the class of $\Gamma$ in $\kontgraphs$, i.e.~we have a well defined homomorphism
\begin{align*}
    \coeffof : \kontgraphs &\to \comps \\
     [\Gamma] &\mapsto c_\Gamma
\end{align*}

\subsubsection{Partial integration}\label{sec:integration_notation}
The proof of our main result involves building graphs by appending certain subgraphs to each other, and saying what that operation does to the integral corresponding to the graph.
To express this induction we generalize the notion of integrating a Kontsevich graph,
where we also allow the external nodes of the graph to vary in the upper half plane
and we only integrate a subset of internal nodes, resulting in a function of the remaining nodes.
The induction then proceeds by integrating out subgraphs in the order they were appended.

We formalize this integration as follows.

\begin{define}\label{def:pushforward-notation}
Let $p : \closure\conf_{T\disunion X}^X \to \closure\conf_{T}$ be the map that forgets the points in $X$.
Then, for $\graphform\in \closure\calv_{T\disunion X}^X \tensor \arnolds^\bullet(\closure\conf_{T\disunion X}^X)$, let
\begin{align*}
    \int_{X} \graphform= p_*(\graphform)
\end{align*}
be the pushforward (fibrewise integral) along $p$.
\end{define}
We have $\int_{X} \graphform\in \closure\calv_{T} \tensor \arnolds^\bullet(\closure\conf_{T})$ by \cite[Theorem 4.1]{Banks_2020};
i.e.~it is a form on $\closure\conf_{T}$ built from logarithmic derivatives $\dd\log(-)$ and polylogarithms that are functions of the locations of the remaining marked points $T$.

\begin{define}\label{def:integration_notation}
Let $\Gamma$ be a Kontsevich graph with internal vertices labeled by $P$ and external vertices labeled by $Q$, and let $T = P\disunion Q$.
We denote by 
\begin{align*}
    \int\Gamma := \int_P \graphform_\Gamma
\end{align*}
the integral of all interior nodes of $\Gamma$, viewed as a function on $\closure\conf_{Q}$.
\end{define}
Note for the stratum $U = \emptyset$, $R = Q$ then this is usual integral defining $c_\Gamma$, and using the isomorphism $\calv_{\emptyset, \{A,B\}} = \mzvs$ (\Cref{ex:calv_ab_is_mzvs}) we express $c_\Gamma$ as an integer linear combination of normalized MZVs.
At the opposite extreme, if $U=Q$, $R = \emptyset$, we treat the external vertices of $\Gamma$ as points $A,B$ moving in the upper half plane and obtain a polylogarithm on $\conf_{R,\emptyset}$.

\begin{ex}
\label{ex:ladder_integration_notation_pt1}
Consider the Kontsevich graph constructed by appending a single rung ladder to itself:
\[
\Gamma_1 = \begin{tikzpicture}[baseline=(current bounding box.center)]
    \begin{scope}[decoration={
    markings,
    mark=at position 1 with {\arrow{>}}}
    ] 
    \node (X) at (0,0) {$x$};
    \node (Y) at (1,0) {$y$};
    \node (U) at (0,-1) {$A_1$};
    \node (V) at (1,-1) {$B_1$};
    \draw[postaction=decorate] (X) to[bend left] (Y);
    \draw[postaction=decorate] (Y) to[bend left] (X);
    \draw[postaction=decorate] (X) -- (U);
    \draw[postaction=decorate] (Y) -- (V);
\end{scope}
\end{tikzpicture}
\qquad\qquad
\Gamma_2 = \begin{tikzpicture}[baseline=(current bounding box.center)]
    \begin{scope}[decoration={
    markings,
    mark=at position 1 with {\arrow{>}}}
    ] 
    \node (X) at (0,0) {$A_1$};
    \node (Y) at (1,0) {$B_1$};
    \node (U) at (0,-1) {$A_2$};
    \node (V) at (1,-1) {$B_2$};
    \draw[postaction=decorate] (X) to[bend left] (Y);
    \draw[postaction=decorate] (Y) to[bend left] (X);
    \draw[postaction=decorate] (X) -- (U);
    \draw[postaction=decorate] (Y) -- (V);
\end{scope}
\end{tikzpicture}
\qquad\qquad
\Gamma = \begin{tikzpicture}[baseline=(current bounding box.center)]
    \begin{scope}[decoration={
    markings,
    mark=at position 1 with {\arrow{>}}}
    ] 
    \node (X) at (0,0) {$x$};
    \node (Y) at (1,0) {$y$};
    \node (U) at (0,-1) {$A_1$};
    \node (V) at (1,-1) {$B_1$};
    \node (A) at (0,-2) {$A_2$};
    \node (B) at (1,-2) {$B_2$};
    \draw[postaction=decorate] (X) to[bend left] (Y);
    \draw[postaction=decorate] (Y) to[bend left] (X);
    \draw[postaction=decorate] (X) -- (U);
    \draw[postaction=decorate] (Y) -- (V);
    \draw[postaction=decorate] (U) to[bend left] (V);
    \draw[postaction=decorate] (V) to[bend left] (U);
    \draw[postaction=decorate] (U) -- (A);
    \draw[postaction=decorate] (V) -- (B);
\end{scope}
\end{tikzpicture}.
\]
Since the $\Gamma_1$ is connected to $\Gamma_2$ only via $A_1,B_1$ 
it follows from the definition of $\graphform_\Gamma$ that (writing the graph in place of the differential form)
\begin{align*}
\graphform_\Gamma = \graphform_{\Gamma_2}\wedge\graphform_{\Gamma_1} = 
\begin{tikzpicture}[baseline=(current bounding box.center)]
    \begin{scope}[decoration={
    markings,
    mark=at position 1 with {\arrow{>}}}
    ] 
    \node (X) at (0,0) {$A_1$};
    \node (Y) at (1,0) {$B_1$};
    \node (U) at (0,-1) {$A_2$};
    \node (V) at (1,-1) {$B_2$};
    \draw[postaction=decorate] (X) to[bend left] (Y);
    \draw[postaction=decorate] (Y) to[bend left] (X);
    \draw[postaction=decorate] (X) -- (U);
    \draw[postaction=decorate] (Y) -- (V);
\end{scope}
\end{tikzpicture}
\bigwedge
\begin{tikzpicture}[baseline=(current bounding box.center)]
    \begin{scope}[decoration={
    markings,
    mark=at position 1 with {\arrow{>}}}
    ] 
    \node (X) at (0,0) {$x$};
    \node (Y) at (1,0) {$y$};
    \node (U) at (0,-1) {$A_1$};
    \node (V) at (1,-1) {$B_1$};
    \draw[postaction=decorate] (X) to[bend left] (Y);
    \draw[postaction=decorate] (Y) to[bend left] (X);
    \draw[postaction=decorate] (X) -- (U);
    \draw[postaction=decorate] (Y) -- (V);
\end{scope}
\end{tikzpicture},
\end{align*}
and by the definition of $c_\Gamma$ we have with $U = \{x,y,A_1,B_1\}$, $R =  \{A_2,B_2\}$ that
\begin{align*}
c_\Gamma  &=  \int_{\conf_{U,R}}
\begin{tikzpicture}[baseline=(current bounding box.center)]
    \begin{scope}[decoration={
    markings,
    mark=at position 1 with {\arrow{>}}}
    ] 
    \node (X) at (0,0) {$A_1$};
    \node (Y) at (1,0) {$B_1$};
    \node (U) at (0,-1) {$A_2$};
    \node (V) at (1,-1) {$B_2$};
    \draw[postaction=decorate] (X) to[bend left] (Y);
    \draw[postaction=decorate] (Y) to[bend left] (X);
    \draw[postaction=decorate] (X) -- (U);
    \draw[postaction=decorate] (Y) -- (V);
\end{scope}
\end{tikzpicture}\bigwedge
\begin{tikzpicture}[baseline=(current bounding box.center)]
    \begin{scope}[decoration={
    markings,
    mark=at position 1 with {\arrow{>}}}
    ] 
    \node (X) at (0,0) {$x$};
    \node (Y) at (1,0) {$y$};
    \node (U) at (0,-1) {$A_1$};
    \node (V) at (1,-1) {$B_1$};
    \draw[postaction=decorate] (X) to[bend left] (Y);
    \draw[postaction=decorate] (Y) to[bend left] (X);
    \draw[postaction=decorate] (X) -- (U);
    \draw[postaction=decorate] (Y) -- (V);
\end{scope}
\end{tikzpicture}.
\end{align*}

This formulation suggests we can break down the integration into two steps: integrate $x$ and $y$, then integrate $A_1$ and $B_1$.
The first step is:
\begin{align*}
    h_1(A_1,B_1) := \int\Gamma_1 = \int_{\{x,y\}}
\begin{tikzpicture}[baseline=(current bounding box.center)]
    \begin{scope}[decoration={
    markings,
    mark=at position 1 with {\arrow{>}}}
    ] 
    \node (X) at (0,0) {$x$};
    \node (Y) at (1,0) {$y$};
    \node (U) at (0,-1) {$A_1$};
    \node (V) at (1,-1) {$B_1$};
    \draw[postaction=decorate] (X) to[bend left] (Y);
    \draw[postaction=decorate] (Y) to[bend left] (X);
    \draw[postaction=decorate] (X) -- (U);
    \draw[postaction=decorate] (Y) -- (V);
\end{scope}
\end{tikzpicture}.
\end{align*}
where we are left with a function $h_1$ on $\closure\conf_{\{A_1,B_1\}}$,
i.e., $A_1$ and $B_1$ may vary in the upper half plane or $\reals$.
Then, for the second integration,
\begin{align*}
    h_2(A_2,B_2) = \int h(A_1,B_1)\Gamma_2 = \int_{\{A_1,B_1\}} h(A_1,B_1)
\begin{tikzpicture}[baseline=(current bounding box.center)]
    \begin{scope}[decoration={
    markings,
    mark=at position 1 with {\arrow{>}}}
    ] 
    \node (X) at (0,0) {$A_1$};
    \node (Y) at (1,0) {$B_1$};
    \node (U) at (0,-1) {$A_2$};
    \node (V) at (1,-1) {$B_2$};
    \draw[postaction=decorate] (X) to[bend left] (Y);
    \draw[postaction=decorate] (Y) to[bend left] (X);
    \draw[postaction=decorate] (X) -- (U);
    \draw[postaction=decorate] (Y) -- (V);
\end{scope}
\end{tikzpicture}
\end{align*}
integrates out $A_1,B_1$ and is now a function on $\closure\conf_{\{A_2,B_2\}}$. By evaluating $h_2$ at $(0,1)$ (i.e., by considering the restriction to the stratum $\conf_{\emptyset, \{A_2, B_2\}}$) and using the isomorphism given in \Cref{ex:calv_ab_is_mzvs}, we obtain $c_\Gamma$.
Notice that each step of this process evaluated the integral of the same ladder subgraph times some function;
hence, given a formula for such an integral, we may evaluate the integral of any sized ladder by iterating that formula (see \Cref{ex:hgraph_single_rung}).
\end{ex}

\section{Nautical polylogarithms and their integrals}

In this section we define a new class of polylogarithms and show that some Kontsevich integrals
are particularly easy to evaluate in terms of these polylogarithms.

\subsection{The algebra of nautical polylogarithms}
\begin{define}
    Let $a, b \in U \disunion R_\infty$, and let $u_1,\dots,u_n \in U\disunion R$ 
    where both (1) $a \in U$ or $a \neq u_n$, and (2) $b \in U$ or $b \neq u_1$ occur.
    Then the polylogarithm $\hyperlog{\conj a}{\conj b}{u_1\cdots u_n}$ with path between $\conj a$ and $\conj b$ 
    is \emph{nautical}.
\end{define}

Given some configuration $U \hookrightarrow \uhp$ and $R \hookrightarrow \reals$,
we can visualize the geometric setup of these polylogarithms as a path in the lower half plane,
with marked points in the upper half plane. 
The name ``nautical'' is derived from the following analogy, pictured in \Cref{fig:naut_example}:
think of the lower half plane as the ocean, the path as the path of a ship, and the upper half plane as the sky with the marked points serving as stars.
\begin{figure}
    \centering
\begin{tikzpicture}
\begin{scope}[decoration={
    markings,
    mark=at position 0.5 with {\arrow{>}}}
    ] 
    \node at (-1, 0) (R) {$\reals$};
    \node at (-1, 1)  {$\uhp$};
    \node at (-1, -1) {$\lhp$};
    \draw (R) -- (4,0); 
    \node at (1,1) {$u_1^{\;\times}$};
    \node at (1.7,0.3) {$u_2$};
    \node at (1.7,0) {$\times$};
    \node at (3,1.2) {$u_3^{\;\times}$};
    
    \node at (-0.2, 0.3) (A) {$\conj a=a$};
    \node at (-0.2, 0) (A) {$\times$};
    \node at (2.5, -0.6) (B) {$\conj b$};
    \draw [postaction={decorate}] (-0.2,0) to[out=270, in=210] (B);
\end{scope}
\end{tikzpicture}
    \caption{
The setup of $\comps$ for $\hyperlog{\conj a}{\conj b}{u_1au_2u_3}$, where $a, u_2 \in R$ and $u_1, u_3 \in U$.
}
    \label{fig:naut_example}
\end{figure}

Since these polylogarithms have no marked points in $\lhp$ and the path is in $\lhp \union \{\conj a, \conj b\}$, the homotopy class of the path is unambiguous,
and so nautical polylogarithms are single-valued.
In addition, by requiring that $a\neq u_n$ when $a\in R$ and $b\neq u_1$ when $b\in R$ we ensure these polylogarithms are always convergent.
We can thus regard these objects as global sections of $\calv_{U,R}$
and more generally, use them to define global sections of $\closure\calv_{T}^X$
when marked points are allowed to vary between $\uhp$ and $\reals$:
\begin{define}
The \emph{algebra of nautical polylogarithms} is the subring of $H^0(\calv_{U,R})$ generated by nautical polylogarithms.
In addition, $\closure\nauts_T^X \subset H^0(\closure\calv_{T}^X)$ is the direct sum of $\nauts_{U,R}$ over all strata of $\closure\conf_T^X$.
\end{define}
Nautical polylogarithms in general are holomorphic with respect to all $u_i \neq a, b$,
and real analytic with respect to $a$ and $b$.
If $a,b \neq u_i$ for all $i$ then the nautical polylogarithm $\hyperlog{\conj a}{\conj b}{u_1\cdots u_n}$ is an antiholomorphic function of $a$ and $b$,
or equivalently a holomorphic function of $\conj a$ and $\conj b$.

The following Lemma will aid in integrating these polylogarithms.
\begin{lem}
    Let $\hyperlog{\conj a}{\conj x}{u_1\cdots u_n} \in \nauts_{U,R}$ with $x \in U$ and $n \geq 1$. Then
    \begin{align*}
        \conj\partial \hyperlog{\conj a}{\conj x}{u_1\cdots u_n} = \conj\omega_{u_1}(x) L(\conj a | u_2\cdots u_n | \conj x)
    \end{align*}
    where $\conj\partial$ is the Dolbeault differential with respect to $\conj x$ in the fibres $\disk_{U,R}$. 
    \label{lem:ftc_for_nauts}
\end{lem}
\begin{proof}
Using the definition of the iterated integral, we have
\begin{align*}
     \hyperlog{\conj a}{\conj x}{u_1\cdots u_n} = \int_{\conj a}^{\conj x} \omega_{u_1}\cdots\omega_{u_n} =  \int_{\conj a}^{\conj x} \omega_{u_1}(t)\int_{\conj a}^{\conj t}\omega_{u_2}\cdots\omega_{u_n}.
\end{align*}
Then since $u_1,\dots, u_n \in U \disunion R$ and $\conj x \in \conj U$ we have $\conj x \neq u_j$ for any $j$, so the fundamental theorem of calculus implies
\begin{align*}
    \conj\partial L(\conj a | u_1\cdots u_n | \conj x) = \omega_{u_1}(\conj x)\int_{\conj a}^{\conj x}\omega_{u_2}\cdots\omega_{u_n} = \conj\omega_{u_1}(x)\hyperlog{\conj a}{\conj x}{u_2\cdots u_n}
\end{align*}
as desired.
\end{proof}

\subsection{Tools for integrating nautical polylogarithms}
We now consider the simplest case of a pushforward, in which we integrate out a single variable, i.e.~we consider a pushforward along the projection, 
\[
\closure \conf_{T\disunion \{x\}}^{\{x\}} \to \closure\conf_T
\]
which is exactly the universal disk
\[
\closure \disk_T \to \closure\conf_T,
\]
with fibre coordinate $x$.

More precisely, we consider an integral $\int_{x} \hyperlog{\conj a}{\conj x}{w} \conj\omega_{u}(x)\wedge\omega_{v}(x)$,
where $\hyperlog{\conj a}{\conj x}{w} \in \closure\nauts_{T\disunion\{x\}}^{\{x\}}$ is a nautical polylogarithm, where $u \in T$, and $v \in S$.
We compute this integral in each stratum $\conf_{U\disunion\{x\},R}\subset\closure\conf_{T\disunion\{x\}}^{\{x\}}$
by identifying $\conf_{U\disunion \{x\},R}$ with the universal disk $\disk_{U,R} \to \conf_{U,R}$, and integrating over the fibres with coordinate $x$.
These kinds of integrals arise when integrating Kontsevich graphs,
and it is easy to construct  a nautical $(1,0)$ primitive of the integrand using \Cref{lem:ftc_for_nauts}:
\begin{align*}
    \conj\partial\hyperlog{\conj a}{\conj x}{u\cdot w}\omega_v(x) = \hyperlog{\conj a}{\conj x}{w}\conj\omega_u(x)\wedge\omega_v(x).
\end{align*}
Then we will be left to compute $\int_{\partial \disk_{U,R}}\hyperlog{\conj a}{\conj x}{u\cdot w}\omega_v(x)$ by Stokes' Theorem. 
The result will sometimes still be nautical, which we will exploit in \Cref{sec:construct_mzvs}.

The main tool used for computing the boundary integral is \Cref{prop:general_residue_thm} below,
which requires a preliminary definition.
\begin{define}
    Suppose $p \in S$, and (under a configuration $S \hookrightarrow \comps$) let $\epsilon_r(p)$ be a counterclockwise loop about $p$ of radius $r > 0$, cut to the left of $p$:
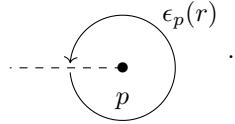
\[\begin{tikzpicture}[baseline=(current bounding box.center)]
\node at (0,0) [label=below:$p$]{$\bullet$};
\draw [->]([shift=(-175:0.7)] 0, 0) arc (-175:175:0.7);
\node at (0.9,0.7) {$\epsilon_p(r)$};
\draw [dashed](0,0) -- (-1.5,0);
\end{tikzpicture}.
\]
   Let $V \subseteq \comps \setminus S$ be an open subset whose closure contains $p$, and suppose that $\alpha$ is a real analytic (1,0) form defined on $V$ that admits an analytic continuation to a form $\tilde\alpha$ on a disc centred at $p$ with a branch cut to the left. 
    The \emph{normalized residue} of $\alpha$ at $p$ is
    \begin{align*}
        \NRes_p \alpha = \lim_{r\to 0}\int_{\epsilon_r(p)}\tilde\alpha.
    \end{align*}
    Note that $\NRes_p \alpha = 2\pi i \Res_p \alpha$ when the usual residue $\Res_p\alpha$ is defined.
\end{define}

\begin{prop}
    Let $\alpha$ be a single-valued polylogarithmic (1,0)-form defined on $\disk_{U,R}$,
    and suppose that $\beta$ is a single-valued holomorphic (1,0)-form defined over $\conjspace\disk_{U,R}$
    such that $\alpha$ and $\beta$ agree along $\reals$,
    i.e., the pullbacks of $\alpha$ and $\beta$ under $\univreals_R \hookrightarrow \partial\disk_{U,R}$ and $\univreals_R \hookrightarrow \partial\conjspace\disk_{U,R}$ agree
    when $\disk_{U,R}$ and $\conjspace\disk_{U,R}$ are naturally embedded in $\rsphere_{U,R}$.
    Then
    \begin{align}
        \int_{\partial \disk_{U,R}} \alpha = -\sum_{p \in R_\infty \disunion \conj U} \NRes_{p} \beta - \sum_{p\in U}\NRes_{p}\alpha \label{eqn:general_residue_thm}
    \end{align}
\label{prop:general_residue_thm}
\end{prop}
\begin{proof}
We can embed the fibres of $\disk_{U,R}$ and $\conjspace\disk_{U,R}$ into $\rsphere_{U,R}$ as the punctured upper and lower half planes of $\comps\projs^1$ and compute the integral in $\comps\projs^1$.
The components of $\partial\disk_{U,R}$ come in three types:
(1) tiny clockwise loops about each $p\in U$,
(2) the collection of intervals $\reals \setminus R$, and
(3) clockwise tiny half-loops around for each point $r \in R_\infty$.
By definition, the integral over the components of type (1) is equal to $-\sum_{p\in U}\NRes_p\alpha$, accounting for the last term in the statement.

Next, by assumption, $\beta$ is defined over the punctured lower half plane,
and $\alpha = \beta$ on $\reals\setminus R$.
It is clear that the integral along the intervals of type (2) can be computed using $\beta$,
and Proposition 4.19 in \cite{Banks_2020} says that the integrals about half loops of type (3) can also be computed using $\beta$, extending by analytic continuation into $\uhp$.
We obtain the following integration path for $\beta$:
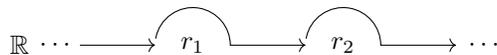
\[
\begin{tikzpicture}
    \node at (-1.8,0) {$\reals$};
    \node at (-1.3, 0) {$\cdots$};
    \draw [->](-1, 0) -- (0,0);
    
    \node at (0.5,0) {$r_1$} ;
    \draw [->]([shift=(170:0.5)] 0.5, 0) arc (170:0:0.5) -- (2,0);
    
    \node at (2.5,0) {$r_2$} ;
    \draw [->]([shift=(170:0.5)] 2.5, 0) arc (170:0:0.5) -- (4,0);
    
    \node at (4.4, 0) {$\cdots$};
\end{tikzpicture}
\]
where the small gaps in the integration path indicate breaks in the analytic continuation of $\beta$ into $\uhp$.
This path of integration is equivalent to a path $\gamma$ traversing all of $\reals$ with half loops in the \emph{lower} half plane, plus clockwise loops at each point $r \in R_\infty$:

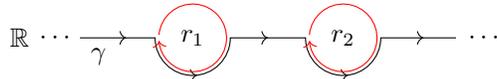
\[
\begin{tikzpicture}
    \node at (-1.8,0) {$\reals$};
    \node at (-1.3, 0) {$\cdots$};
\begin{scope}[decoration={
    markings,
    mark=at position 0.1 with {\arrow{>}},
    mark=at position 0.3 with {\arrow{>}},
    mark=at position 0.5 with {\arrow{>}},
    mark=at position 0.72 with {\arrow{>}},
    mark=at position 0.93 with {\arrow{>}},
    }] 
    \draw[postaction={decorate}] (-1, 0) -- 
      ([shift=(180:0.5)] 0.5, 0) arc (180:360:0.5) -- 
      (2,0) --  
      ([shift=(180:0.5)] 2.5, 0) arc (180:360:0.5) -- 
      (4,0);
\end{scope}
    \node at (-0.75, -0.25) {$\gamma$};
    
    \draw[red, ->] ([shift=(175:0.45)] 0.5, 0) arc (175:-175:0.45);
    \node at (0.5,0) {$r_1$} ;

    \draw[red, ->] ([shift=(175:0.45)] 2.5, 0) arc (175:-175:0.45);
    \node at (2.5,0) {$r_2$} ;
    
    \node at (4.4, 0) {$\cdots$};
\end{tikzpicture}
\]

The clockwise loops now contribute the terms $-\sum_{p\in R_\infty} \NRes_p\beta$ in \Cref{eqn:general_residue_thm}.
Since $\beta$ is holomorphic in $\lhp \setminus \conj U$, a deformation of the loop $\gamma$ (which encloses $\lhp$) to small clockwise loops about each point $p \in \conj U$
accounts for the remaining $-\sum_{p\in \conj U} \NRes_p\beta$ in \Cref{eqn:general_residue_thm}.
\end{proof}

Most nautical polylogarithms naturally satisfy the hypotheses of \Cref{prop:general_residue_thm} when their endpoint is the variable of integration, as
illustrated by the following example.
\begin{ex}
\label{ex:finding_lhp_form}
Consider the following polylogarithms:
    \[
\begin{tikzpicture}
\begin{scope}[decoration={
    markings,
    mark=at position 0.5 with {\arrow{>}}}
    ] 
    \begin{scope}[shift={(0,0)}]
        \draw (-8,0) -- (-4,0); 
        \node at (-6.5,1) {$u^{\;\times}$};
        \node at (-6,0.3) {$1$};
        \node at (-6,0) {$\times$};
        
        \node at (-7.2, 0.3) (A) {$0$};
        \node at (-5,-0.6) (B) {$\conj x$};
        \node at (-5,0.6) {$x^{\;\times}$};
        \draw [postaction={decorate}] (-7.2,0) to[out=270, in=200] (B);

        \node at (-6,-1.5) {$\alpha(x) = \hyperlog{0}{\conj x}{ux1}\dd x$};
    \end{scope}

    \begin{scope}[shift={(5,0)}]
        \draw (-8,0) -- (-4,0); 
        \node at (-6.5,1) {$u^{\;\times}$};
        \node at (-6,0.3) {$1$};
        \node at (-6,0) {$\times$};
        
        \node at (-7.2, 0.3) (A) {$0$};
        \node at (-5,0.3) (B) {$x = \conj x$};
        \node at (-5,0) {$\times$};
        \draw [postaction={decorate}] (-7.2,0) to[out=270, in=270] (-5,0);
        
        \node at (-6,-1.5) {$\alpha|_\reals = \hyperlog{0}{x}{ux1}\dd x = \beta|_\reals$};
    \end{scope}
    
    \begin{scope}[shift={(10,0)}]
        \draw (-8,0) -- (-4,0); 
        \node at (-6.5,1) {$u^{\;\times}$};
        \node at (-6,0.3) {$1$};
        \node at (-6,0) {$\times$};
        
        \node at (-7.2, 0.3) (A) {$0$};
        \node at (-4.7,-0.6) {$\conj x$};
        \node at (-5,-0.6) {$\times$};
        \draw [postaction={decorate}] (-7.2,0) to[out=270, in=200] (-5,-0.6);
        
        \node at (-6,-1.5) {$\beta(\conj x) = \hyperlog{0}{\conj x}{u\conj x1}\dd \conj x$};
    \end{scope}
\end{scope}
\end{tikzpicture}
\]
    The nautical polylogarithmic form $\alpha(x) = \hyperlog{0}{\conj x}{ux1}\dd x$ has $u \in U$ with labels $0,1 \in R$ identified with $0,1\in\comps$.
    The restriction of $\alpha(x)$ to $\reals \setminus \{1\}$ is given by $\hyperlog{0}{x}{ux1}\dd x$ where the path from 0 to $x$ is always in $\lhp \union \{0, x\}$, matching the choice of path for $\alpha$.

    This restricted form agrees with the holomorphic (1,0)-form $\beta(x) = \hyperlog{0}{x}{u x 1}\dd x$ along $\reals \setminus \{1\}$
    where, $x$ is thought to be a holomorphic coordinate in the $\emph{lower}$ half plane.
    Since $x$ is conceptually a coordinate on the upper half plane in the fibres of $\disk_{U,R}$,
    we equivalently state that
    the $\alpha$ equals the form $\beta(\conj x) = \hyperlog{0}{\conj x}{u \conj x 1}\dd \conj x$ along $\reals \setminus \{1\}$,
    which remains holomorphic as a function of $\conj x$
    as it ranges over $\lhp$, i.e., ranges over $\conjspace\disk_{U,R}$.
    Despite $\conj x$ being a letter in the polylogarithm's integrand and so potentially introducing  monodromy, we establish that $\beta$ is single-valued over $\conjspace\disk_{U,R}$ in \Cref{lem:lhp_restriction_is_sv} below.
\end{ex}
Note however that not all nautical polylogarithms restrict so naturally;
for instance, $\hyperlog{0}{\conj x}{xu}$ is no longer nautical when $x$ is restricted to $\reals$ (and in fact diverges).
Our proofs will avoid such cases.

\begin{lem}
    \label{lem:lhp_restriction_is_sv}
    Let $a \in U \disunion R_\infty$, and let $s_1,\dots,s_n \in U \disunion R \disunion \{\conj x\}$ with  $s_1 \neq \conj x$.
    Then the polylogarithm $\hyperlog{\conj a}{\conj x}{s_1\cdots s_n}$ with integration path contained in the fibres of $\conjspace\disk_{U,R}$ is single-valued as $\conj x$ ranges in a fibre of $\conjspace\disk_{U,R}$.
\end{lem}
\begin{proof}
Since the path is in $\conjspace\disk_{U,R}$ and all letters in the integrand are in $U$ or $R$ except $\conj x$,
the only possible branch point is at $\conj x$.
Let $\epsilon$ be a loop about $\conj x$, and $\eta$ a path from $\conj a$ to $\conj x$.
Then the monodromy of the polylogarithm in question is computed by $\int_{\epsilon * \eta} u_1\cdots u_n - \int_{\eta} u_1\cdots u_n$.
By the path concatenation formula \Cref{eqn:path_concat} we have
\begin{align*}
    \int_{\epsilon * \eta} u_1\cdots u_n - \int_{\eta} u_1\cdots u_n = \sum_{i=1}^{n}\int_\epsilon u_1\cdots u_i\int_\eta u_{i+1}\cdots u_n.
\end{align*}
Since $u_1 \neq \conj x$, all integrals about $\epsilon$ vanish by \cite[Proposition 2.14]{goncharov2001multiple},
so we have no monodromy as desired.
\end{proof}

\begin{remark}
This is a special case that occurs since the only branch point is at an endpoint of the polylogarithm,
so the monodromy calculation only sees prefixes of the word.
\end{remark}

In order to compute the normalized residues in \Cref{prop:general_residue_thm}, we establish some analytic results.
Note that
non-zero residues occur only when the loop encircles singular points of $\alpha$ and $\beta$.
We will only need to consider isolated logarithmic singularities, and so we establish results about normalized residues at these points.
In fact, they behave much like the usual residues.

\begin{lem}
    For $k \in \nats_{\geq 0}$ and $j \in \ints$, we have $\NRes_0 z^k(\log z)^j \dd z = 0$ for any choice of branch of $\log z$.
    \label{lem:int_z_times_log}
\end{lem}
\begin{proof}
    Let $\epsilon_r$ be a loop about zero parameterized by $re^{i\theta}$ and let $\log z = \log|z| + i\theta$ be some choice of branch of the complex logarithm where $\theta \in [\tau, \tau+2\pi]$ for some $\tau \in \reals$.
    By the triangle inequality, $\int_{\epsilon_r} |z|^k|\log z|^j \dd z \leq 2\pi r^{k+1}(|\log r| + |\tau|+ |2\pi|)^j$ for all $r > 0$.
    Since $\lim_{r \to 0} r^n (\log r)^j = 0$ for all integers $n > 0$,
    we conclude that $\NRes_{0} z^k(\log z)^j\dd z = \lim_{r\to 0}\int_{\epsilon_r} z^k(\log z)^j\dd z = 0$ for all $k \geq 0$, as desired.
\end{proof}

\begin{lem}
    Suppose $h$ is a holomorphic multivalued function defined in an open punctured disk centred at $p \in \comps$ with an isolated logarithmic singularity at $p$. Then for any choice of branch of $h$, the following statements hold:
    \begin{enumerate}
        \item $\NRes_p h(z)\dd z = 0$
        \item If $\lim_{z\to p}h(z)$ exists, then $\NRes_p(h(z)\omega_p) = \lim_{z\to p}h(z)$ 
    \end{enumerate}
    \label{lem:nres_log_singularities}
\end{lem}
\begin{proof}
        
    Assume by an affine coordinate transformation that $p=0$ and so $h(z) = \sum_{j=0}^n f_j(z) (\log z)^j$, where $f_j(z)$ is holomorphic at $z = 0$
    and $\log z$ denotes some choice of branch of the complex logarithm.
    Replacing the functions $f_j$ with their Taylor expansions, we obtain
    \begin{align*}
        h(z) = \sum_{j=0}^n\sum_{k=0}^{\infty}a_{jk}z^k(\log z)^j.
    \end{align*}
    where $a_{jk} \in \mathbb{C}$ are the Taylor coefficients for $f_j$.
    By uniform convergence of Taylor series on closed disks, and \Cref{lem:int_z_times_log}, we have
    \begin{align*}
        \NRes_0 h(z)\dd z = \sum_{j=0}^n\sum_{k=0}^{\infty}a_{jk}\NRes_0 z^k(\log z)^j \dd z = 0
    \end{align*}
    establishing the first case.
    For the second case, since $\lim_{z\to p}h(z)$ exists, we must have $a_{j0} = 0$ for all $j > 0$.
    Then, since $\omega_0 = \frac{1}{2\pi i}\frac{1}{z}$ we have
    \begin{align*}
        \NRes_0 h(z)\omega_0 &= \frac{1}{2\pi i}\sum_{j=1}^n\sum_{k=1}^{\infty}a_{jk}\NRes_0z^{k-1}(\log z)^j \dd z + \frac{1}{2\pi i}\NRes_0\frac{f_0(z)}{z}\dd z \\
        &= f_0(0) = \lim_{z\to p}h(z)
    \end{align*}
    where we have used \Cref{lem:int_z_times_log}.
\end{proof}

\subsection{Integration of a wedge}\label{sec:wedge_integral}
We are now in a position to integrate some Kontsevich subgraphs using nautical polylogarithms,
These subgraph integrals will give some recursive identities which we will use to establish our main result.
\begin{lem}
\label{lem:wedgeint}
Let $\hyperlog{\overline{y}}{\overline{x}}{w} \in \closure\nauts_{T\disunion\{x\}}^{\{x\}}$ be a nautical polylogarithm with  $w \in\polywords$ a single word, and let $u,v \in T$. 
Suppose that either $y \neq u,v$, or $w$ is empty. Then
\begin{align}
    \int_{x} \hyperlog{\overline{y}}{\overline{x}}{w}
\begin{tikzpicture}[baseline=(current bounding box.center)]
\begin{scope}[decoration={
    markings,
    mark=at position 0.5 with {\arrow{>}}}
    ] 
    \node (X) at (0,0) {$x$};
    \node (U) at (-0.5,-1) {$u$};
    \node (V) at (0.5,-1) {$v$};
    \draw[postaction=decorate] (X) -- (U);
    \draw[postaction=decorate] (X) -- (V);
\end{scope}
\end{tikzpicture} =
    \hyperlog{\overline y}{\overline{v}}{u\cdot w_{x\mapsto v}} - \hyperlog{\overline y}{\overline{u}} {v\cdot w_{x\mapsto u}}
    -  \begin{cases}
        0 & \text{$x$ is in $w$} \\
        \hyperlog{\conj y}{\infty}{(u-v)\cdot w} & \text{otherwise}
    \end{cases} \label{eqn:wedgeint}
\end{align}
where the propagator is ordered $\alpha_{x\to u} \wedge \alpha_{x\to v}$ and $w_{x\mapsto a}$ indicates the word obtained by replacing all instances of $x$ in $w$ with $a$.
\end{lem}
\begin{proof}
Recalling that $\int_x$ is the pushforward under $\closure\conf_{T\disunion\{x\}}^{\{x\}} \to \closure\conf_{T}$ in \Cref{def:pushforward-notation},
we pick some partition $T = U \disunion R$ (as in \Cref{def:closure_conf_space})
and we consider the stratum $\conf_{U\disunion\{x\}, R} \subset \closure\conf_{T\disunion\{x\}}^{\{x\}}$.
We then identify this stratum with the universal disk $\disk_{U, R} \to \conf_{U, R}$
and integrate in the fibres of $\disk_{U,R}$.

We expand the propagator:
\begin{align*}
    \alpha_{x\to u} \wedge \alpha_{x\to v} &= \left(\frac{1}{2\pi i}\right)^2\dd\log\left(\frac{x-u}{\conj x - u}\right)\wedge \dd\log\left(\frac{x-v}{\conj x - v}\right) \\
    &= \conj \omega_v\wedge \omega_u - \conj\omega_u\wedge\omega_v
\end{align*}
By \Cref{lem:ftc_for_nauts}, a $\conj\partial$-primitive for the integrand $\hyperlog{\overline{y}}{\overline{x}}{w}(\conj \omega_v\wedge \omega_u - \conj\omega_u\wedge\omega_v)$ is
\begin{align*}
    \alpha = \hyperlog{\conj y}{\conj x}{v\cdot w}\omega_u - \hyperlog{\conj y}{\conj x}{u\cdot w}\omega_v.
\end{align*}

The strings $v\cdot w$ and $u\cdot w$ do not end with $y$: on the one hand, if $w$ is non-empty it does not end with $y$ since it defines a nautical polylogarithm,
and on the other hand, if $w$ is empty we require $y\neq u,v$ by hypothesis.
Furthermore, $v\cdot w$ and $u\cdot w$ do not start with $x$ since $u,v \in T$ and $x\notin T$ by assumption.
Therefore this primitive is nautical for all $x$ (including when $x$ collides with $\reals$) and so is single-valued.

We apply Stokes' Theorem and integrate $\alpha$ over all boundary components of the fibres of $\disk_{U,R}$.
In these fibres (considering an embedding into $\comps\projs^1$), the form $\alpha(x)$ agrees along $\reals$ with
\begin{align*}
\beta(\conj x) = \hyperlog{\conj y}{\conj x}{v\cdot w_{x \mapsto \conj x}}\conj\omega_u - \hyperlog{\conj y}{\conj x}{u\cdot w_{x \mapsto \conj x}}\conj\omega_v,
\end{align*}
which is holomorphic in $\lhp$.
Here the polylogarithm integration paths are contained in $\lhp \disunion \{\conj x, \conj y\}$, similar to \Cref{ex:finding_lhp_form}.
By \Cref{lem:lhp_restriction_is_sv}, the polylogarithms in $\beta$ are single-valued as $\conj x$ ranges in $\lhp$, so $\beta$ satisfies the hypotheses of \Cref{prop:general_residue_thm}.

For the partitions where $u,v \in U$ we have that
\begin{align*}
    \int_{\partial\disk_{U,R}} \alpha = -\NRes_u\alpha - \NRes_v\alpha - \NRes_\infty\beta.
\end{align*}
The first two terms are straightforward residues at the simple poles of $\omega_u$ and $\omega_v$, giving the first two desired terms:
\begin{align*}
    \NRes_u\alpha = \hyperlog{\overline y}{\overline{u}}{v\cdot w_{x\mapsto u}} && 
    \NRes_v\alpha = -\hyperlog{\overline y}{\overline{v}}{u\cdot w_{x\mapsto v}}.
\end{align*}
To compute $\NRes_\infty\beta$, we make a change of coordinate $\conj z = 1/\conj x$ and take the normalized residue at 0.
Here it is clearest to write the polylogarithms in iterated integral form, and abuse notation by writing $\omega_0w$ to mean the concatenation $0\cdot w \in \polywords$:
\begin{align*}
    \beta(\conj x) &= \omega_{u}(\conj x)\int_{\conj y}^{\conj x}{\omega_{v}w_{x \mapsto \conj x}} - \omega_{v}(\conj x)\int_{\conj y}^{\conj x}{\omega_{u}w_{x \mapsto \conj x}} 
\end{align*}
It is easy to verify that $\omega_a(\conj x) = \omega_{1/a}(\conj z) - \omega_0(\conj z)$ so
\begin{align*}
    \beta(\conj z) &= (\omega_{1/u}(\conj z) - \omega_0(\conj z))\int_{1/\conj y}^{\conj z}{(\omega_{1/v} - \omega_0)w'} - (\omega_{1/v}(\conj z)-\omega_0(\conj z))\int_{1/\conj y}^{\conj z}{(\omega_{1/u}-\omega_0)w'} \\
    &= \omega_{1/u}(\conj z)\int_{1/\conj y}^{\conj z}(\omega_{1/v} - \omega_0)w' - \omega_{1/v}(\conj z)\int_{1/\conj y}^{\conj z}(\omega_{1/u} - \omega_0)w' + \omega_0(\conj z)\int_{1/\conj y}^{\conj z}(\omega_{1/u}-\omega_{1/v})w'
\end{align*}
where $w'$ is the transformed version of $w_{x\mapsto\conj x}$.
The first two terms here have no normalized residue at 0 by \Cref{lem:nres_log_singularities}: any divergence as $z \to 0$ is at worst logarithmic.
For the last polylogarithm, the integrand does not start with $\omega_0$ and so is convergent as $z \to 0$ and
we can apply \Cref{lem:nres_log_singularities}, case 2,
to say that the residue is 
\begin{align*}
    \int_{1/\conj y}^0(\omega_{1/u}-\omega_{1/v})w'_{z\mapsto 0}.
\end{align*}
In the case that $\conj x$ is a letter in $w_{x\mapsto\conj x}$, so that $w_{x\mapsto\conj x} = w_0 \conj x w_1 $ where $w_0, w_1 \in \polywords$, applying the transformation rule we see
$w' = w_0'(\omega_{\conj z} - \omega_0)w_1'$, so the residue is
\begin{align*}
    \int_{1/\conj y}^0(\omega_{1/u}-\omega_{1/v})(w_0'(\omega_{\conj z} - \omega_0)w_1')_{z\mapsto 0} = \int_{1/\conj y}^0(\omega_{1/u}-\omega_{1/v})w_0'(\omega_0 - \omega_0)w_1' = 0.
\end{align*}
Transforming back to the original coordinates and notation, we see that
\begin{align*}
    \NRes_\infty\beta = \begin{cases}
        0 & \text{$x$ is in $w$}\\
        \int_{\conj y}^\infty(\omega_{u}-\omega_{v})w & \text{otherwise} 
    \end{cases}
    = \begin{cases}
        0 & \text{$x$ is in $w$} \\
        \hyperlog{\conj y}{\infty}{(u-v)\cdot w} & \text{otherwise}
    \end{cases},
\end{align*}
giving the last desired term.

For partitions where $u \in U$ and $v\in R$, \Cref{prop:general_residue_thm} instead gives $\int_{\partial \disk_{U,R}} \alpha = -\NRes_v\alpha -\NRes_u\beta - \NRes_\infty\beta$ but we will obtain the same result.
Indeed, since $u = \conj u$ and the residue of $\hyperlog{\conj y}{\conj x}{uw}\conj\omega_v$ vanishes at $u$ by \Cref{lem:nres_log_singularities} case 1, we have $\NRes_u\beta = \hyperlog{\overline y}{\overline{u}}{vw_{x\mapsto u}}$.
The same idea holds if $v \in R$, so the result holds for any partition $U,R$.
\end{proof}

\begin{ex}
\label{ex:single_wedge}
    For the integral of a single wedge, we recover the known integral $-\frac{1}{2}$ \cite{kontsevich_2003}.
    Let $u=0, v=1$, and $y \in \uhp$ be an arbitrary point.
    Then 
    \begin{align*}
        \int_{x}\alpha_{x\to 0}\wedge\alpha_{x\to 1}  &=  \hyperlog{\conj y}{1}{0} - \ \hyperlog{\conj y}{0}{1} - \hyperlog{\conj y}{\infty}{0} + \hyperlog{\conj y}{\infty}{1}\\
        &= \frac{1}{2\pi i}\left(\log\frac{1}{\conj y} - \log\frac{0-1}{\conj y - 1} - \lim_{t\to\infty} \left(\log\left(\frac{t}{\conj y}\right) - \log\left(\frac{t-1}{\conj y - 1}\right)\right)\right) \\
        &= \frac{1}{2\pi i}\log(-1) \\
        &= -\frac{1}{2},
    \end{align*}
    where $\log(-1)=-i\pi$ since the path of integration is in the lower half plane.
\end{ex}

\subsection{Integration of a ladder}\label{sec:ladder_integral}
\begin{lem}
Let $\hyperlog{\conj y}{\conj x}{w} \in \closure\nauts_{T\disunion\{x,y\}}^{\{x,y\}}$ be a nautical polylogarithm and let $u,v\in T$. Then
\begin{align}
    \int_{\{x,y\}}\hyperlog{\conj y}{\conj x}{w}
\begin{tikzpicture}[baseline=(current bounding box.center)]
    \begin{scope}[decoration={
    markings,
    mark=at position 1 with {\arrow{>}}}
    ] 
    \node (X) at (0,0) {$x$};
    \node (Y) at (1,0) {$y$};
    \node (U) at (0,-1) {$u$};
    \node (V) at (1,-1) {$v$};
    \draw[postaction=decorate] (X) to[bend left] (Y);
    \draw[postaction=decorate] (Y) to[bend left] (X);
    \draw[postaction=decorate] (X) -- (U);
    \draw[postaction=decorate] (Y) -- (V);
\end{scope}
\end{tikzpicture}
    = \hyperlog{\conj v}{\conj u}{v\cdot w_{\substack{\\x\mapsto u\\y\mapsto v}}\cdot u} \in \closure\nauts_{T}, \label{eqn:ladder_integral}
\end{align}
where the propagator is ordered $\alpha_{x\to y}\wedge\alpha_{y\to x}\wedge\alpha_{x\to u}\wedge\alpha_{y\to v}$.
\end{lem}
\begin{proof}
We can factor this pushforward over $\closure\conf_{T\disunion\{x,y\}}^{\{x,y\}} \to \closure\conf_{T}$ into two steps,
in particular,
integrate out $x$ by the pushforward over $\closure\conf_{T\disunion\{x,y\}}^{\{x,y\}} \to \closure\conf_{T\disunion\{y\}}^{\{y\}}$
and then integrate out $y$ by the pushforward over $\closure\conf_{T\disunion\{y\}}^{\{y\}} \to \closure\conf_{T}$.
Similar to the proof of \Cref{lem:wedgeint},
to compute the first pushforward consider some partition $T\disunion\{y\} = U\disunion R$ and where $y \in U$ and integrate in the fibres of $\disk_{U,R}$.
We then view these fibres under some embedding in $\comps\projs^1$ to compute the integral.

First, we expand the propagator to a useful form to construct a primitive with respect to $\conj x$:
\begin{gather*}
\alpha_{x\to y}\wedge\alpha_{y\to x}\wedge\alpha_{x\to u}\wedge\alpha_{y\to v} = \conj\omega_y\wedge(\omega_u-\omega_{\conj y}) \wedge \alpha_1
    + \conj\omega_u\wedge(\omega_y - \omega_{\conj y})\wedge \alpha_2
\end{gather*}
where
\begin{align*}
    \alpha_1 &= \alpha_{y\to u} \wedge \alpha_{y\to v} & \textrm{and} && \alpha_2 &= \dd\log\left(\frac{y-v}{y-u}\right)\wedge \dd\log(\conj y-v)
\end{align*}
are independent of $x$ or $\conj x$.
This expansion was derived using the software provided by Banks--Panzer--Pym, which expands propagators into a basis of relative forms in $\arnolds^\bullet$ (see \cite[Section 2.4.2]{Banks_2020}) 
and can be verified manually.

Then, by \Cref{lem:ftc_for_nauts}, a single-valued $\conj\partial$-primitive for the integrand in \Cref{eqn:ladder_integral} is
\begin{align*}
    \alpha = \hyperlog{\conj y}{\conj x}{y\cdot w} (\omega_u-\omega_{\conj y}) \wedge \alpha_1 + \hyperlog{\conj y}{\conj x}{u\cdot w}(\omega_y - \omega_{\conj y})\wedge \alpha_2
\end{align*}
which is nautical for all values of $x$ (including when $x$ collides with $\reals$, or if $w$ is empty),
since $y\in U$ and $x\neq y,u$ by assumption.
Then $\alpha$ agrees with following form on $\reals$:
\begin{align*}
    \beta(\conj x) = \hyperlog{\conj y}{\conj x}{y\cdot w_{x\mapsto \conj x}} (\conj\omega_u-\conj\omega_{\conj y}) \wedge \alpha_1 + \hyperlog{\conj y}{\conj x}{u\cdot w_{x\mapsto \conj x}}(\conj\omega_y - \conj\omega_{\conj y})\wedge \alpha_2
\end{align*}
where the path is always taken to be in $\lhp\union\{\conj x\}$.
In contrast to the wedge scenario, we now have a pole in $\lhp$, at $\conj y$.
By \Cref{lem:lhp_restriction_is_sv}, $\beta$ is a holomorphic one-form on $\lhp  \union \reals \setminus \{u, v, \conj y\}$,
so the conditions for \Cref{prop:general_residue_thm} are now satisfied to integrate $\alpha$ over the boundary.

For partitions where $u \in U$ we have
\begin{align*}
    \int_{\partial \disk_{U,R}} \alpha = -\NRes_u\alpha - \NRes_y\alpha - \NRes_{\conj y}\beta - \NRes_\infty\beta
\end{align*}
and we will obtain the same results if $u \in R$ since $\alpha$ and $\beta$ have the same residue there (applying \Cref{lem:nres_log_singularities} where needed).

The first three residues are simple:
\begin{align*}
    \NRes_u\alpha = \hyperlog{\overline y}{\overline{u}}{y\cdot w_{x\mapsto u}} \alpha_1, &&
    \NRes_y\alpha = 0, &&\textrm{and} &&
    \NRes_{\conj y}\beta = 0,
\end{align*}
noting that two residues are zero because $\hyperlog{\overline y}{\overline y}{\cdot} = 0$ since
the iterated integration path has no length.

For $\NRes_\infty\beta$, note that after transforming $\beta$ with the coordinate change $\conj z = 1/\conj x$, both $\conj\omega_u - \conj\omega_{\conj y}$ and $\conj\omega_y - \conj\omega_{\conj y}$ will have no pole at 0, so \Cref{lem:nres_log_singularities} case 1 implies the normalized residue is zero.
We conclude that
\begin{align*}
    \int_{x}\hyperlog{\conj y}{\conj x}{w}(\alpha_{x\to y}\wedge\alpha_{y\to x}\wedge\alpha_{x\to u}\wedge\alpha_{y\to v}) = -\hyperlog{\conj y}{\conj u}{y\cdot w_{x\mapsto u}} \alpha_{y\to u}\wedge\alpha_{y\to v}.
\end{align*}

We are left with the pushforward integrating $y$.
Note that $\hyperlog{\conj y}{\conj u}{y\cdot w_{x\mapsto u}} \in \closure\nauts_{T\disunion\{y\}}^{\{y\}}$ is nautical (even when $w$ is empty),
so we have the integral of a nautical polylogarithm times a wedge.
We can thus reverse the path of integration of the polylogarithm, apply \Cref{lem:wedgeint}, and then reverse the path again.  Doing so results in an overall sign flip, since the iterated integral defining the polylogarithm increases in length by one between the two path reversals. 

This yields the final result.
\end{proof}

\begin{ex} \label{ex:hgraph_single_rung}
Consider the graphs
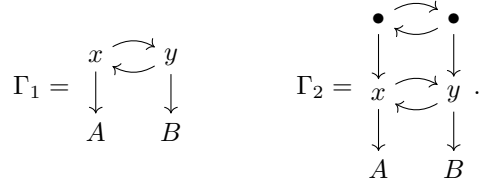
\[
\Gamma_1 = \begin{tikzpicture}[baseline=(current bounding box.center)]
    \begin{scope}[decoration={
    markings,
    mark=at position 1 with {\arrow{>}}}
    ] 
    \node (U) at (0,-1) {$x$};
    \node (V) at (1,-1) {$y$};
    \node (A) at (0,-2) {$A$};
    \node (B) at (1,-2) {$B$};
    \draw[postaction=decorate] (U) to[bend left] (V);
    \draw[postaction=decorate] (V) to[bend left] (U);
    \draw[postaction=decorate] (U) -- (A);
    \draw[postaction=decorate] (V) -- (B);
\end{scope}
\end{tikzpicture}
\qquad\qquad
\Gamma_2 = \begin{tikzpicture}[baseline=(current bounding box.center)]
    \begin{scope}[decoration={
    markings,
    mark=at position 1 with {\arrow{>}}}
    ] 
    \node (X) at (0,0) {$\bullet$};
    \node (Y) at (1,0) {$\bullet$};
    \node (U) at (0,-1) {$x$};
    \node (V) at (1,-1) {$y$};
    \node (A) at (0,-2) {$A$};
    \node (B) at (1,-2) {$B$};
    \draw[postaction=decorate] (X) to[bend left] (Y);
    \draw[postaction=decorate] (Y) to[bend left] (X);
    \draw[postaction=decorate] (X) -- (U);
    \draw[postaction=decorate] (Y) -- (V);
    \draw[postaction=decorate] (U) to[bend left] (V);
    \draw[postaction=decorate] (V) to[bend left] (U);
    \draw[postaction=decorate] (U) -- (A);
    \draw[postaction=decorate] (V) -- (B);
\end{scope}
\end{tikzpicture}.
\]
For $\Gamma_1$, the integral of the top two nodes is easily computed as $\hyperlog{\conj A}{\conj B}{AB}$ using \Cref{eqn:ladder_integral}.
At $A=0, B=1$ we see that $c_{\Gamma_1} = \hyperlog{0}{1}{01} = -\zeta(2)/(2\pi i)^2 = \frac{1}{24}$.

We may integrate $\Gamma_2$ by first integrating the top two nodes and then $x$ and $y$,
in the same procedure as described in \Cref{ex:ladder_integration_notation_pt1}.
The top two nodes will integrate to $\hyperlog{\conj x}{\conj y}{xy}$ times a single rung ladder:
\begin{align*}
    \int\Gamma_2 = \int_{\{x,y\}}\hyperlog{\conj x}{\conj y}{xy}\begin{tikzpicture}[baseline=(current bounding box.center)]
    \begin{scope}[decoration={
    markings,
    mark=at position 1 with {\arrow{>}}}
    ] 
    \node (X) at (0,0) {$x$};
    \node (Y) at (1,0) {$y$};
    \node (U) at (0,-1) {$A$};
    \node (V) at (1,-1) {$B$};
    \draw[postaction=decorate] (X) to[bend left] (Y);
    \draw[postaction=decorate] (Y) to[bend left] (X);
    \draw[postaction=decorate] (X) -- (U);
    \draw[postaction=decorate] (Y) -- (V);
\end{scope}
\end{tikzpicture}
\end{align*}
We can then re-apply \Cref{eqn:ladder_integral} with to obtain $\int\Gamma_2 = \hyperlog{A}{B}{AABB}$.
Evaluating at $A = 0, B = 1$, we see that $c_{\Gamma_2} = \hyperlog{0}{1}{0011} = \frac{1}{(2\pi i)^4}\zeta(1,3) = \frac{1}{5760}$.

By repeating this procedure, we see that the coefficient associated to graph $\Gamma$ given by  a ladder with $n$ rungs will be
\begin{align*}
    c_\Gamma = \hyperlog{0}{1}{0^n1^n} = \frac{(-1)^n}{(2\pi i)^{2n}}\zeta(1,\dots,1,n+1),
\end{align*}
where there are $n-1$ repeated 1s.
\end{ex}


\section{Construction and integration of some Kontsevich graphs}\label{sec:construct_mzvs}

In this section, we define a recursive procedure to both construct and integrate particular Kontsevich graphs, namely, those that can be constructed from a single rung ladder and applying a sequence of the following operations:
\begin{enumerate}
    \item appending a wedge to the left/right external nodes in two possible ways, which we will call the $\prepend$ and $\append$ operations, or
    \item taking two graphs in the set and identifying their external nodes, which we will denote by $*$.
\end{enumerate}
We parameterize the construction of these graphs by the space of all possible compositions of the functions $\prepend,\append$ and $*$, with a single-rung graph as the initial arguments.

Then the key results of this section are how these operations affect the integral of a graph in the sense  of \Cref{def:integration_notation}.
Namely, in \Cref{lem:can_append} below, we show that if $\Gamma$ is a graph obtained from the construction above, with external nodes denoted $A$ and $B$, and $\int\Gamma = \hyperlog{\conj A}{\conj B}{w}$, then  $\int \prepend(\Gamma) = \hyperlog{\conj A}{\conj B}{A\cdot w}$ and $\int \append(\Gamma) = \hyperlog{\conj A}{\conj B}{w\cdot B}$. Similarly, in \Cref{lem:can_shuffle}, we show that if the graphs $\Gamma_1$ and $\Gamma_2$ have $\int \Gamma_1 = \hyperlog{\conj A}{\conj B}{w_1}$ and $\int \Gamma_2 = \hyperlog{\conj A}{\conj B}{w_2}$
then $\int \Gamma_1 * \Gamma_2 = \hyperlog{\conj A}{\conj B}{w_1\shuffle w_2}$.  
Combined, these two Lemmas give a simple algorithm to compute the integral of all graphs constructed in this way.

\subsection{Constructing graphs}

We formalize our graph constructions by endowing the space $\kontgraphs$ of graphs with a multiplication operator (joining)
and operators $\prepend, \append$ that append wedges to a graph.

\begin{define}
The \emph{join} of two Kontsevich graphs $\Gamma_1, \Gamma_2 \in \kontgraphs$ is the graph obtained by identifying the endpoints of the two graphs:
\[
\Gamma_1 * \Gamma_2 := 
\left(
\begin{tikzpicture}[baseline=(current bounding box.center)]
\begin{scope}[decoration={
    markings,
    mark=at position 1 with {\arrow{>}}}
    ] 
    \node (A) at (0,0) {$A$};
    \node (B) at (1,0) {$B$};
    \draw[postaction=decorate,dashed] (0,1) to (A);
    \draw[postaction=decorate,dashed] (1,1) to (B);
\end{scope}
\end{tikzpicture}
\right)
*
\left(
\begin{tikzpicture}[baseline=(current bounding box.center)]
\begin{scope}[decoration={
    markings,
    mark=at position 1 with {\arrow{>}}}
    ] 
    \node (A) at (0,0) {$A$};
    \node (B) at (1,0) {$B$};
    \draw[postaction=decorate,dashed] (0,1) to (A);
    \draw[postaction=decorate,dashed] (1,1) to (B);
\end{scope}
\end{tikzpicture}
\right)
=
\begin{tikzpicture}[baseline=(current bounding box.center)]
\begin{scope}[decoration={
    markings,
    mark=at position 1 with {\arrow{>}}}
    ] 
    \node (A) at (0,0) {$A$};
    \node (B) at (1,0) {$B$};
    \draw[postaction=decorate,dashed] (-1,1) to (A);
    \draw[postaction=decorate,dashed] (0,1) to (B);
    \draw[postaction=decorate,dashed] (1,1) to (A);
    \draw[postaction=decorate,dashed] (2,1) to (B);
\end{scope}
\end{tikzpicture}
\]
The ordering of $\Gamma_1*\Gamma_2$ is given by the ordering of edges in $\Gamma_1$ followed by the ordering of edges of $\Gamma_2$ so that $\graphform_{\Gamma_1*\Gamma_2} = \graphform_{\Gamma_1}\wedge \graphform_{\Gamma_2}$.
\end{define}

\begin{lem}
    Joining graphs is commutative and associative.
\end{lem}
\begin{proof}
    The result is obvious for graphs without an orientation.
    With oriented graphs, the result follows since the wedge product of even-degree forms is commutative.
\end{proof}

We thus consider $\kontgraphs$ as a commutative ring where multiplication is given on generators by joining graphs.

\begin{define}
    Let $\appendring = \ints\gen{\prepend,\append}$ be the free (non-commutative) ring with two generators $\prepend, \append$.
\end{define}

We then endow $\kontgraphs$ with an additional $\appendring$-module structure encoding the action of appending wedges.
For $\Gamma \in \kontgraphs$, let $\prepend(\Gamma)$ and $\append(\Gamma)$ be the graphs obtained by adding a wedge to $\Gamma$ in the following ways:
\[
\prepend\left(
\begin{tikzpicture}[baseline=(current bounding box.center)]
\begin{scope}[decoration={
    markings,
    mark=at position 1 with {\arrow{>}}}
    ] 
    \node (A) at (0,0) {$A$};
    \node (B) at (1,0) {$B$};
    \draw[postaction=decorate,dashed] (0,1) to (A);
    \draw[postaction=decorate,dashed] (1,1) to (B);
\end{scope}
\end{tikzpicture}
\right) = 
\begin{tikzpicture}[baseline=(current bounding box.center)]
\begin{scope}[decoration={
    markings,
    mark=at position 1 with {\arrow{>}}}
    ] 
    \node (A) at (0,-1) {$A$};
    \node (B) at (1,0) {$B$};
    \node (C) at (1,-1) {$B'$};
    \draw[postaction=decorate,dashed] (0,0.66) to (A);
    \draw[postaction=decorate,dashed] (1,0.66) to (B);
    \draw[postaction=decorate] (B) to (A);
    \draw[postaction=decorate] (B) to (C);
\end{scope}
\end{tikzpicture}
\quad\quad\quad
\append\left(
\begin{tikzpicture}[baseline=(current bounding box.center)]
\begin{scope}[decoration={
    markings,
    mark=at position 1 with {\arrow{>}}}
    ] 
    \node (A) at (0,0) {$A$};
    \node (B) at (1,0) {$B$};
    \draw[postaction=decorate,dashed] (0,1) to (A);
    \draw[postaction=decorate,dashed] (1,1) to (B);
\end{scope}
\end{tikzpicture}
\right) = 
\begin{tikzpicture}[baseline=(current bounding box.center)]
\begin{scope}[decoration={
    markings,
    mark=at position 1 with {\arrow{>}}}
    ] 
    \node (A) at (0,0) {$A$};
    \node (B) at (1,-1) {$B$};
    \node (C) at (0,-1) {$A'$};
    \draw[postaction=decorate,dashed] (0,0.66) to (A);
    \draw[postaction=decorate,dashed] (1,0.66) to (B);
    \draw[postaction=decorate] (A) to (B);
    \draw[postaction=decorate] (A) to (C);
\end{scope}
\end{tikzpicture}
\]
where the new ordering of edges is given by appending the new edges to the original order alphabetically, so that $\omega_{\prepend(\Gamma)} = \graphform_\Gamma \wedge \alpha_{B\to A} \wedge \alpha_{B\to B'}$ and $\omega_{\append(\Gamma)} = \graphform_\Gamma \wedge \alpha_{A\to A'} \wedge \alpha_{A\to B}$.
We then extend $\prepend$ and $\append$ $\ints$-linearly to make $\kontgraphs$ an $\appendring$-module.

We can now consider all graphs that are constructed using joins and appending wedges as above.
A parameterization of these graphs is given by the following space:

\begin{define}
    Let $\syntaxtrees$ be the free abelian group generated by rooted trees with nodes labelled $\bullet, e, \prepend,\append$, and $*$, where $\bullet$ is the root, $e$ is the label of all other leaf nodes, nodes labelled $\prepend$ or $\append$ have degree two, and nodes labelled $*$ have degree three.
    We equip $\syntaxtrees$ with a commutative ring and $\appendring$-module structure by the relations
\begin{align*}
\begin{gathered}\begin{tikzpicture}[baseline=(current bounding box.center)]
    \def\ysep{1}
    \node (A) at (0,1*\ysep) {$g_1$};
    \node (R) at (0,0*\ysep) {$\bullet$};
    \draw[](A)--(R);
\end{tikzpicture}\end{gathered}
*
\begin{gathered}\begin{tikzpicture}[baseline=(current bounding box.center)]
    \begin{scope}[shift={(0,0)}]
        \def\ysep{1}
        \node (A) at (0,1*\ysep) {$g_2$};
        \node (R) at (0,0*\ysep) {$\bullet$};
        \draw[](A)--(R);
    \end{scope}
\end{tikzpicture}\end{gathered}
=
\begin{gathered}\begin{tikzpicture}[baseline=(current bounding box.center)]
    \begin{scope}[shift={(0,0)}]
        \def\ysep{1.5}
        \node (A) at (-0.5,1*\ysep) {$g_1$};
        \node (B) at (0.5,1*\ysep) {$g_2$};
        \node (S) at (0,0.5*\ysep) {$*$};
        \node (R) at (0,0*\ysep) {$\bullet$};
        \draw[](A)--(S)--(B);
        \draw(S)--(R);
    \end{scope}
\end{tikzpicture}\end{gathered}
&&
\prepend\cdot
\begin{gathered}\begin{tikzpicture}[baseline=(current bounding box.center)]
    \def\ysep{1}
    \node (A) at (0,1*\ysep) {$g$};
    \node (R) at (0,0*\ysep) {$\bullet$};
    \draw[](A)--(R);
\end{tikzpicture}\end{gathered}
=
\begin{gathered}\begin{tikzpicture}[baseline=(current bounding box.center)]
    \def\ysep{1.5}
    \node (A) at (0,1*\ysep) {$g$};
    \node (B) at (0,0.5*\ysep) {$\prepend$};
    \node (R) at (0,0*\ysep) {$\bullet$};
    \draw[](A)--(B);
    \draw(B)--(R);
\end{tikzpicture}\end{gathered}
&&
\append\cdot
\begin{gathered}\begin{tikzpicture}[baseline=(current bounding box.center)]
    \def\ysep{1}
    \node (A) at (0,1*\ysep) {$g$};
    \node (R) at (0,0*\ysep) {$\bullet$};
    \draw[](A)--(R);
\end{tikzpicture}\end{gathered}
=
\begin{gathered}\begin{tikzpicture}[baseline=(current bounding box.center)]
    \def\ysep{1.5}
    \node (A) at (0,1*\ysep) {$g$};
    \node (B) at (0,0.5*\ysep) {$\append$};
    \node (R) at (0,0*\ysep) {$\bullet$};
    \draw[](A)--(B);
    \draw(B)--(R);
\end{tikzpicture}\end{gathered}
\end{align*}
along with the relations induced by assuming that $*$ is commutative and associative.
$\syntaxtrees$ is also equipped with a grading where the weight of a tree is $2\# e + \#\prepend + \#\append$, where $\#$ is the number of times that node appears in the tree (see below for why we double count $e$ nodes).
Let $\syntaxtrees^n$ be the subgroup of all weight $n$ trees so that $\syntaxtrees = \bigoplus_{n=0}^{\infty}\syntaxtrees^n$ is graded by weight.
\end{define}

\begin{remark}
    By starting at the root node,
    each tree in $\syntaxtrees$ can be represented as some composition of the $\prepend, \append, *$ operations followed by evaluating the composition with $e$ in all arguments.
    We will subsequently write elements of $\syntaxtrees$ in this notation.
    Some example elements are shown below.
\[
\begin{tikzpicture}
\begin{scope}[main/.style = {}] 
    \def\labeloffset{-0.6}
    \begin{scope}[shift={(0,0)}]
        \def\ysep{1}
        \node[main] (R) at (0,0*\ysep) {$\bullet$};
        \node[main] (A) at (0,1*\ysep) {$*$};
        \node[main] (B) at (-0.5,2*\ysep) {$e$};
        \node[main] (C) at (0.5,2*\ysep) {$e$};
        \draw (B) -- (A) -- (C);
        \draw (A) -- (R);
        \node at (0,\labeloffset) {$e*e$};
    \end{scope}

    \begin{scope}[shift={(4,0)}]
        \def\ysep{0.7}
        \node[main] (R) at (0,0*\ysep) {$\bullet$};
        \node[main] (A) at (0,1*\ysep) {$\prepend$};
        \node[main] (B) at (0,2*\ysep) {$\append$};
        \node[main] (C) at (0,3*\ysep) {$e$};
        \draw (R) -- (A) -- (B) -- (C);
        \node at (0,\labeloffset) {$\prepend(\append(e))$};
    \end{scope}

    \begin{scope}[shift={(8,0)}]
        \def\ysep{0.6}
        \node[main] (A1) at (-0.5,4*\ysep) {$e$};
        \node[main] (A2) at (0.5,4*\ysep) {$e$};
        \node[main] (B1) at (-0.5,3*\ysep) {$\append$};
        \node[main] (B2) at (0.5,3*\ysep) {$\prepend$};
        \node[main] (C) at (0,2*\ysep) {$*$};
        \node[main] (D) at (0,1*\ysep) {$\prepend$};
        \node[main] (R) at (0,0*\ysep) {$\bullet$};
        \draw (A1) -- (B1) -- (C) -- (D) -- (R);
        \draw (A2) -- (B2) -- (C);
        \node at (0,\labeloffset) {$\prepend(\append(e) * \prepend(e))$};
    \end{scope}
    
\end{scope}
\end{tikzpicture}
\]

\end{remark}

\begin{remark}
    $\syntaxtrees$ can also be defined as the free algebra generated by a single element $e$ over the operad $\text{Comm}\gen{\prepend, \append}$ where $\prepend$ and $\append$ are unary operations.
\end{remark}

Given this parameter space, we can now define how to construct graphs using elements of $\syntaxtrees$ by defining a morphism to $\kontgraphs$:
\begin{define}
    We denote by $\graphof : \syntaxtrees \to \kontgraphs$ the unique $\appendring$-module and ring homomorphism such that $\graphof(e)$ is the single ladder rung in \Cref{ex:hgraph_single_rung}.
\end{define}
\begin{remark}
The choice that $e$ nodes have weight 2 in $\syntaxtrees$ is natural since $e$ maps to a graph with two internal nodes.
We choose $e$ to map to this graph so that the conditions for \Cref{lem:can_append} later are always met.
\end{remark}

\begin{ex}\label{ex:append_and_join_graph}
    Let $t = \append(e)*\prepend(e) \in \syntaxtrees$. Then
\[
\graphof(\append(e)*\prepend(e)) = 
\left(
\begin{tikzpicture}[baseline=(current bounding box.center)]
\begin{scope}[decoration={
    markings,
    mark=at position 1 with {\arrow{>}}}
    ] 
    \node (A) at (0,0) {$A$};
    \node (B) at (1,0) {$B$};
    \node (C) at (0,1) {$\bullet$};
    \node (D) at (0,2) {$\bullet$};
    \node (E) at (1,2) {$\bullet$};
    \draw[postaction=decorate] (D) to[bend left] (E);
    \draw[postaction=decorate] (E) to[bend left] (D);
    \draw[postaction=decorate] (D) to (C);
    \draw[postaction=decorate] (E) to (B);
    \draw[postaction=decorate] (C) to (A);
    \draw[postaction=decorate] (C) to (B);
\end{scope}
\end{tikzpicture}
\right)*\left(
\begin{tikzpicture}[baseline=(current bounding box.center)]
\begin{scope}[decoration={
    markings,
    mark=at position 1 with {\arrow{>}}}
    ] 
    \node (A) at (0,0) {$A$};
    \node (B) at (1,0) {$B$};
    \node (C) at (1,1) {$\bullet$};
    \node (D) at (0,2) {$\bullet$};
    \node (E) at (1,2) {$\bullet$};
    \draw[postaction=decorate] (D) to[bend left] (E);
    \draw[postaction=decorate] (E) to[bend left] (D);
    \draw[postaction=decorate] (E) to (C);
    \draw[postaction=decorate] (D) to (A);
    \draw[postaction=decorate] (C) to (A);
    \draw[postaction=decorate] (C) to (B);
\end{scope}
\end{tikzpicture}
\right)
= 
\begin{tikzpicture}[baseline=(current bounding box.center)]
\begin{scope}[decoration={
    markings,
    mark=at position 1 with {\arrow{>}}}
    ] 
    \node (A) at (0,0) {$A$};
    \node (B) at (1,0) {$B$};
    \node (C) at (-1,1) {$\bullet$};
    \node (C2) at (2,1) {$\bullet$};
    \node (D) at (-1,2) {$\bullet$};
    \node (E) at (0,2) {$\bullet$};
    \node (F) at (1,2) {$\bullet$};
    \node (G) at (2,2) {$\bullet$};
    \draw[postaction=decorate] (D) to[bend left] (E);
    \draw[postaction=decorate] (E) to[bend left] (D);
    \draw[postaction=decorate] (F) to[bend left] (G);
    \draw[postaction=decorate] (G) to[bend left] (F);
    \draw[postaction=decorate] (D) to (C);
    \draw[postaction=decorate] (E) to (B);
    \draw[postaction=decorate] (G) to (C2);
    \draw[postaction=decorate] (C2) to (A);
    \draw[postaction=decorate] (C2) to (B);
    \draw[postaction=decorate] (F) to (A);
    \draw[postaction=decorate] (C) to (A);
    \draw[postaction=decorate] (C) to (B);
\end{scope}
\end{tikzpicture}
\]
\end{ex}

\subsection{Integrating graphs}

Using the formulae developed in \Cref{sec:wedge_integral} and \Cref{sec:ladder_integral},
we are able to easily integrate graphs in $\graphof(\syntaxtrees)$.
In particular, we establish formulae for how the operations $\prepend$, $\append$, and $*$
change the polylogarithm $\int\Gamma$,
which gives an easy recursive way to evaluate $\int G(t)$ for any $t \in \syntaxtrees$.

This integration procedure stays entirely in the space of nautical polylogarithms $\closure\nauts_{\{A,B\}}$.
Since the integrands of these polylogarithms are just strings of $A$s and $B$s,
we introduce a corresponding space of binary strings:

\begin{define}
Let $\allstrings$ be the free abelian group (written additively) generated by binary strings starting with $0$ and ending with $1$.
Endow $\allstrings$ with a commutative ring structure given by the shuffle product $\shuffle$ 
and an $\appendring$-module structure given by $\prepend(s) = 0\cdot s$ and $\append(s) = s\cdot 1$.
This group also carries a grading by the length of the string so that $\allstrings = \bigoplus_{n=0}^{\infty}\allstrings^n$, where $\allstrings^n$ is the subgroup of weight $n$ strings.
\end{define}

We shall then write $\hyperlog{\conj A}{\conj B}{w}$ for $w \in \allstrings$ to mean the nautical polylogarithm obtained by exchanging $0$ with $A$ and $1$ with $B$.
For instance, $\hyperlog{\conj A}{\conj B}{01101}$ corresponds to $\hyperlog{\conj A}{\conj B}{ABBAB}$. 
The $\appendring$-module structure on $\allstrings$ corresponds to that on $\kontgraphs$:

\begin{lem}
Suppose that $\Gamma \in \kontgraphs$ and let $w \in \allstrings$ be a word where 0 and 1 both appear at least once.
Suppose $\int\Gamma = \hyperlog{\conj A}{\conj B}{w}$. 
Then $\int\prepend(\Gamma) = \hyperlog{\conj A}{\conj B}{\prepend(w)}$ and $\int\append(\Gamma) = \hyperlog{\conj A}{\conj B}{\append(w)}$.
\label{lem:can_append}
\end{lem}
\begin{proof}
For $\prepend$, label the nodes as follows:
\[
\prepend\left(
\begin{tikzpicture}[baseline=(current bounding box.center)]
\begin{scope}[decoration={
    markings,
    mark=at position 1 with {\arrow{>}}}
    ] 
    \node (A) at (0,0) {$A$};
    \node (B) at (1,0) {$y$};
    \draw[postaction=decorate,dashed] (0,1) to (A);
    \draw[postaction=decorate,dashed] (1,1) to (B);
\end{scope}
\end{tikzpicture}
\right) = 
\begin{tikzpicture}[baseline=(current bounding box.center)]
\begin{scope}[decoration={
    markings,
    mark=at position 1 with {\arrow{>}}},
    ] 
    \node (A) at (0,-1) {$A$};
    \node (B) at (1,0) {$y$};
    \node (C) at (1,-1) {$B$};
    \draw[postaction=decorate,dashed] (0,0.66) to (A);
    \draw[postaction=decorate,dashed] (1,0.66) to (B);
    \draw[postaction=decorate] (B) to (A);
    \draw[postaction=decorate] (B) to (C);
\end{scope}
\end{tikzpicture}
\]
Let $X$ be the set of all nodes above $A$ and $y$.
By definition, $\graphform_{\prepend(\Gamma)} = \graphform_\Gamma\wedge\alpha_{y\to A}\wedge\alpha_{y\to B}$, so
\begin{align*}
    \int\prepend(\Gamma) = \int_{y}\alpha_{y\to A}\wedge\alpha_{y\to B}\int_{X}\graphform_\Gamma = \int_{y}\hyperlog{\conj A}{\conj y}{w}\alpha_{y\to A}\wedge\alpha_{y\to B},
\end{align*}
since we can integrate out nodes in any order.
Then by the wedge integral \Cref{eqn:wedgeint}:
\begin{align*}
    \int_{y} \hyperlog{\conj A}{\conj y}{w} \alpha_{y\to A}\wedge\alpha_{y\to B} = \hyperlog{\conj A}{\conj y}{0\cdot w} = \hyperlog{\conj A}{\conj y}{\prepend(w)}
\end{align*}
where the $\hyperlog{\conj A}{\infty}{(A-B)\cdot w}$ term is not needed since $0$ is a letter of $w$ and the $\hyperlog{\conj A}{\conj A}{1\cdot w}$ term vanishes due to the zero length path between $\conj A$ and $\conj A$.

The proof of the formula for $\int \append(\Gamma)$ follows by a similar argument with a path reversal.
\end{proof}

Then, the shuffle algebra structure in $\allstrings$ corresponds to the $*$ operation in $\kontgraphs$:

\begin{lem}
    Suppose $\Gamma_1, \Gamma_2 \in \kontgraphs$ are such that $\int\Gamma_1 = \hyperlog{\conj A}{\conj B}{w_1}$ and $\int\Gamma_2 = \hyperlog{\conj A}{\conj B}{w_2}$ with $w_1, w_2 \in \allstrings$.
    Then $\int\Gamma_1 * \Gamma_2 = \hyperlog{\conj A}{\conj B}{w_1\shuffle w_2}$
    with $w_1\shuffle w_2 \in \allstrings$.
    \label{lem:can_shuffle}
\end{lem}
\begin{proof}
We have $\graphform_{\Gamma_1 *\Gamma_2} = \graphform_{\Gamma_1}\wedge\graphform_{\Gamma_2}$
and so with $X_1, X_2$ denoting the internal nodes of $\Gamma_1, \Gamma_2$ respectively:
\begin{align*}
    \int\Gamma_1*\Gamma_2 &= \int_{X_1\disunion X_2} \graphform_{\Gamma_1}\wedge\graphform_{\Gamma_2} \\&= \int_{X_1} \graphform_{\Gamma_1}\int_{X_2} \graphform_{\Gamma_2} \\
    &= \hyperlog{\conj A}{\conj B}{w_1}\hyperlog{\conj A}{\conj B}{w_2} \\
    &= \hyperlog{\conj A}{\conj B}{w_1\shuffle w_2}
\end{align*}
by Fubini's Theorem, and the shuffle product for iterated integrals \Cref{eqn:shuffle_prod}.
Since since $w_1$ and $w_2$ do not start with $1$ or end with $0$, their shuffles cannot either, so $w_1\shuffle w_2 \in \allstrings$.
\end{proof}

Using \Cref{lem:can_append} and \Cref{lem:can_shuffle}, we can thus compute $\int\Gamma$ for any graph $\Gamma \in \graphof(\syntaxtrees)$ by transforming the integrand of the nautical polylogarithm in order of the operations in $t$, starting with the formula $\int \graphof(e) = \hyperlog{\conj A}{\conj B}{01}$ derived in \Cref{ex:integrating_graph_with_nauts}.
Note that by starting with the single rung ladder $\graphof(e)$ we ensure that the condition in \Cref{lem:can_append}, that 0 and 1 are both in $w$, is always satisfied.

After all integration steps have been performed, we are left with a polylogarithm $L(\conj A|w|\conj B)$ where $w$ is a string of $A$s and $B$s.
Setting $A=0$ and $B=1$ and using the relation
\begin{align*}
    \hyperlog{0}{1}{0^{n_d-1}1\cdots 10^{n_1-1}1} = (-1)^d\frac{\zeta(n_1,\dots,n_d)}{(2\pi i)^{n_1+\dots+n_d}} \in \mzvs
\end{align*}
due to Le--Murakami \cite{mzvrepn1995} and Kontsevich, we obtain an expression for $c_\Gamma$ as a linear combination of MZVs associated to binary strings.

\begin{ex}\label{ex:integrating_graph_with_nauts}
    Consider the element $\prepend(e) \in \syntaxtrees$ which has weight three.
    The associated graph $\Gamma = \graphof(\prepend(e))$ is:
    \[
\begin{tikzpicture}[baseline=(current bounding box.center)]
\begin{scope}[decoration={
    markings,
    mark=at position 1 with {\arrow{>}}}
    ] 
    \node (A) at (0,0) {$A$};
    \node (B) at (1,0) {$B$};
    \node (C) at (1,1) {$\bullet$};
    \node (D) at (0,2) {$\bullet$};
    \node (E) at (1,2) {$\bullet$};
    \draw[postaction=decorate] (D) to[bend left] (E);
    \draw[postaction=decorate] (E) to[bend left] (D);
    \draw[postaction=decorate] (E) to (C);
    \draw[postaction=decorate] (D) to (A);
    \draw[postaction=decorate] (C) to (A);
    \draw[postaction=decorate] (C) to (B);
\end{scope}
\end{tikzpicture}
\]
Using \Cref{lem:can_append}, we have
\begin{align*}
    \int\Gamma = \int \graphof(\prepend(e)) = \int \prepend(\graphof(e)) = \hyperlog{\conj A}{\conj B}{\prepend(01)} = \hyperlog{\conj A}{\conj B}{001}.
\end{align*}
Identifying $A$ and $B$ with the the points $0, 1 \in \reals$
we obtain $c_{\Gamma} = -\frac{\zeta(3)}{(2\pi i)^3}$.
\end{ex}

Since the only object that changes during the integration steps above is the integrand of the polylogarithm,
we can encode this procedure using the following morphisms:

\begin{define}
   We denote by $\stringof : \syntaxtrees \to \allstrings$ the unique $\appendring$-module and ring homomorphism such that $\stringof(e) = 01$.
\end{define}

\begin{define}
    We denote by $L^{01} : \allstrings \to \mzvs$ the $\ints$-linear map defined on generators by
    \begin{align*}
    0^{n_d-1}1\cdots 10^{n_1-1}1 \mapsto (-1)^d\frac{\zeta(n_1,\dots,n_d)}{(2\pi i)^{n_1+\dots+n_d}} \in \mzvs
\end{align*}
\end{define}

Note that $L^{01}$ is a ring homomorphism by \Cref{eqn:shuffle_prod},
however it is $\emph{not}$ an $\appendring$-module homomorphism.
For example, it can be verified that $\zeta(1,2) = \zeta(3)$ so
$L^{01}(011) = -L^{01}(001)$. However, $L^{01}(\prepend(011)) \neq -L^{01}(\prepend(001))$ since
$L^{01}(0011) = (2\pi i)^{-4}\zeta(1,3) \neq (2\pi i)^{-4}\zeta(4) = -L^{01}(0001)$.

We also recall the map $\coeffof : \kontgraphs \to \comps$ where $\Gamma \mapsto c_\Gamma$, given by evaluating $\int\Gamma$ at $(A,B) = (0,1)$.
We can now state the main result of this section, a method for obtaining the coefficients $c_\Gamma$ associated to the graphs in $\graphof(\syntaxtrees)$:
\begin{prop}\label{prop:integration_procedure}
    For all $t \in \syntaxtrees$, we have $\coeffof(\graphof(t)) = L^{01}(\stringof(t))$.
    All morphisms respect the grading of each space.
    In other words, the diagram
\[
\begin{tikzcd}
\syntaxtrees^n \arrow[r, "\graphof"] \arrow[d, "\stringof"] & \kontgraphs^n \arrow[d, "\coeffof"] \\
\allstrings^n \arrow[r, "L^{01}"]     & \mzvswithhalf^n
\end{tikzcd}
\]
    commutes for all $n$.
\end{prop}

\begin{proof}
    That all morphisms between these spaces respect the grading are either established results or easy verifications.

    The main result is clear by induction on the weight of the graph: the base case is the single rung ladder in \Cref{ex:hgraph_single_rung} and the induction proceeds by considering the operations $\prepend, \append, *$ in $\syntaxtrees$ to build larger graphs and using \Cref{lem:can_append} and \Cref{lem:can_shuffle} to see that the larger graph is integrated correctly.
    Again, by starting with a ladder, we ensure that the hypotheses of \Cref{lem:can_append} are always satisfied.
\end{proof}

\begin{ex}
Consider $\Gamma = \graphof(\prepend(e)) \in \kontgraphs$, the same as in \Cref{ex:integrating_graph_with_nauts}.
Through the $\stringof$ morphism we can directly compute 
\begin{align*}
c_\Gamma = L^{01}(\stringof(\prepend(e))) = L^{01}(\prepend(01)) = L^{01}(001) = -\frac{\zeta(3)}{(2\pi i)^3}
\end{align*}
matching the previous result.\end{ex}

\begin{ex}
Consider $t = \append(e)*\prepend(e) \in \syntaxtrees$. Its corresponding weight six graph $\Gamma = \graphof(t)$ is the one shown in \Cref{ex:append_and_join_graph}:
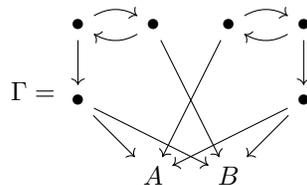
\[
\Gamma = 
\begin{tikzpicture}[baseline=(current bounding box.center)]
\begin{scope}[decoration={
    markings,
    mark=at position 1 with {\arrow{>}}}
    ] 
    \node (A) at (0,0) {$A$};
    \node (B) at (1,0) {$B$};
    \node (C) at (-1,1) {$\bullet$};
    \node (C2) at (2,1) {$\bullet$};
    \node (D) at (-1,2) {$\bullet$};
    \node (E) at (0,2) {$\bullet$};
    \node (F) at (1,2) {$\bullet$};
    \node (G) at (2,2) {$\bullet$};
    \draw[postaction=decorate] (D) to[bend left] (E);
    \draw[postaction=decorate] (E) to[bend left] (D);
    \draw[postaction=decorate] (F) to[bend left] (G);
    \draw[postaction=decorate] (G) to[bend left] (F);
    \draw[postaction=decorate] (D) to (C);
    \draw[postaction=decorate] (E) to (B);
    \draw[postaction=decorate] (G) to (C2);
    \draw[postaction=decorate] (C2) to (A);
    \draw[postaction=decorate] (C2) to (B);
    \draw[postaction=decorate] (F) to (A);
    \draw[postaction=decorate] (C) to (A);
    \draw[postaction=decorate] (C) to (B);
\end{scope}
\end{tikzpicture}
\]
We then have
\begin{align*}
    c_\Gamma &= L^{01}(\stringof(\append(e)*\prepend(e))) \\
    &= L^{01}(\append(01)\shuffle\prepend(01)) \\
    &= L^{01}( 011 \shuffle 001) \\
    &= L^{01}(9\cdot 000111 + 5\cdot 001011 + 2\cdot 001101 + 2\cdot 010011 + 010101 + 011001) \\
    &= -\frac{1}{(2\pi i)^6}(9\zeta(1,1,4) + 5\zeta(1,2,3) + 2 \zeta(2,1,3) + 2\zeta(1,3,2) + \zeta(2,2,2)+\zeta(3,1,2)) 
\end{align*}

\end{ex}

\section{Generating MZVs}
We now prove our main result that $\mzvswithhalf^n \tensor \qs$ is spanned by graphs of weight $n$ and $\mzvswithhalf$ is spanned by all graphs.
Note that since $L^{01}(\allstrings^n) = \mzvs^n$ and by the commutativity of the diagram in \Cref{prop:integration_procedure}, it suffices to show that $\stringof(\syntaxtrees^n\tensor\qs) = \allstrings^n\tensor\qs$.
A key step to this proof is establishing a Lyndon word polynomial basis for $\allstrings$, the topic of the next section.

\subsection{Lyndon words and generation of shuffle algebras}

In this section we study $\allstrings$ as just a shuffle algebra of binary strings,
disregarding the $\appendring$-module structure.
\begin{define}
    A \emph{Lyndon word} is a binary string $w$ such that any prefix of $w$ is lexicographically less than the corresponding suffix,
    that is, for any factorization $w = x\cdot y$ we have $x < y$.
    We adopt the usual lexicographical ordering induced by $1 > 0$.
\end{define}

By a theorem of Radford \cite{radford79}, the Lyndon words are a set of polynomial generators for the shuffle algebra of \emph{all} binary strings over $\qs$. However it is unclear how this result translates to the smaller algebra $\allstrings$.
The main result of this section is \Cref{prop:full_lyndon_decomposition}, that Lyndon words in $\allstrings$ form a set of polynomial generators for $\allstrings$ using coefficients in $\qs$.
The following sequence of Lemmas build to this Proposition.
\begin{lem}
    \label{lem:len2_lyndons_have01prop}
    All Lyndon words of length two or more are in $\allstrings$, i.e., start with 0 and end with 1.
\end{lem}
\begin{proof}
If $w$ starts with 1, then let $w = x\cdot y$ where $x$ is the longest prefix of all 1s.
If $w$ is all ones, the decomposition $w = 1\dots1\cdot 1$ shows it is not Lyndon as long as $|w| \geq 2$.
So suppose $y$ is non-empty and starts with a 0. Then $x > y$ and so $w$ is not Lyndon.

If $w$ ends with 0, then $w = x\cdot 0$ has $x \geq 0$ for any $x$ of length 1 or more,
so it is not Lyndon.
\end{proof}

\begin{cor}
    \label{lem:1_is_maximal_lyndon}
    1 is the maximal Lyndon word.
\end{cor}

\begin{define}
    A \emph{Lyndon decomposition} of a binary string $w$
    is a decomposition of $w$ into Lyndon words in decreasing lexicographic order: $w = l_1^{k_1}\cdots l_n^{k_n}$ where $l_i$ are Lyndon words such that $l_i > l_{i+1}$
    for all $i$.
\end{define}

By the Chen--Fox--Lyndon theorem, a Lyndon decomposition of a binary string $w$ always exists and is unique \cite{Chen_Fox_Lyndon, Shi58}.
Then by \cite{Melancon_Reutenauer_1989} with $w$ decomposed as $l_1^{k_1}\cdots l_n^{k_n}$ we have
\begin{align}
\label{eqn:shuffle_of_lyndon_decomp}
    \frac{1}{k_1!\cdots k_n!} l_1^{\shuffle k_1} \shuffle \dots\shuffle l_n^{\shuffle k_n} = w + \sum c_i w_i
\end{align}
where $w_i$ are words of strictly lower lexicographic order and $k_i, c_i \in \nats$.
\begin{lem}
    \label{lem:lyndon_len2_decomp}
    The Lyndon decomposition of a word in $\allstrings$ only uses words in $\allstrings$.
\end{lem}
\begin{proof}
    We need only exclude the words 0 and 1 in the decomposition by \Cref{lem:len2_lyndons_have01prop}.
    Suppose that $w = s_1\cdot0^k\cdot s_2$ in its Lyndon decomposition.
    Since $0$ has the lowest lexicographic ordering, $w$ must end with $0^k$,
    contradicting that $w\in\allstrings$.
    Similarly, suppose that $w = s_1\cdot1^{k}\cdot s_2$ in its Lyndon decomposition.
    By \Cref{lem:1_is_maximal_lyndon}, 1 has the highest lexicographic ordering in the Lyndon words
    so $w$ must start with 1, again contradicting that $w\in\allstrings$.
\end{proof}

\begin{prop}
\label{prop:full_lyndon_decomposition}
    Any word $w \in \allstrings$ can expressed as a $\qs$-linear combination of shuffles of Lyndon words in $\allstrings$.
\end{prop}
\begin{proof}
We proceed by induction on the lexicographic order.
Let $w = l_1^{k_1}\cdots l_n^{k_n}$ be the Lyndon decomposition of $w$.
Then using \Cref{eqn:shuffle_of_lyndon_decomp} we have 
\begin{align*}
    w =  \frac{1}{k_1!\cdots k_n!} l_1^{\shuffle k_1} \shuffle \dots\shuffle l_n^{\shuffle k_n} - \sum c_i w_i
\end{align*}
where $k_i, c_i \in \nats$ and $w_i$ are binary strings such that $w_i < w$.
By \Cref{lem:lyndon_len2_decomp} we have $l_k \in \allstrings$, and it is clear that $\sum_{i} k_i|l_i| = |w|$, and $|w_i| = |w|$.
Also, since $\allstrings$ is closed under $\shuffle$ and $c_i > 0$, we must have that $w_i \in \allstrings$ for all $i$.

We can subsequently rewrite each $w_i$ that is not a Lyndon word using \Cref{eqn:shuffle_of_lyndon_decomp} the same way, that is, write it in terms of shuffles of Lyndon words in $\allstrings$ and words of lower lexicographic order in $\allstrings$.
This process will terminate in finitely many steps since the lexicographic order is always decreasing and there are finitely many strings of length $|w|$. 
\end{proof}

We end this section with an additional lemma about Lyndon words that will be useful for the main proof in the next section:

\begin{lem}
    \label{lem:lyndons_startwith00}
    All Lyndon words of length three or more start with 00 or end with 11
\end{lem}
\begin{proof}
    For length three, the only Lyndon words are 001 and 011.

    So consider some Lyndon word $w = b_0\cdot b_1 \cdot s \cdot b_2 \cdot b_3$, where $b_i \in \{0,1\}$
    and $s$ is a (possibly empty) binary string.
    \Cref{lem:len2_lyndons_have01prop} says $b_0 = 0$ and $b_3 = 1$,
    so we have $w = 0\cdot b_1 \cdot s \cdot b_2 \cdot 1$.
    Then if $b_1 = 0$ or $b_2 = 1$ we have a prefix or suffix of 00 or 11 respectively, consistent with the Lemma,
    so the only remaining case is $w = 01\cdot s \cdot 01$.
    But then the prefix $01\cdot s$ of $w$ is greater than or equal to the suffix $01$,
    so it is not a Lyndon word.
\end{proof}

Note that there are no similar statements involving prefixes and suffixes 000 and 111 or longer, since 001011 is a Lyndon word.

\subsection{Spanning MZVs}

\begin{theorem}
\label{thm:span_weight_n_mzvs_over_q}
    Integrals of Kontsevich graphs of weight $n$, with the log propagator, span the $\qs$-vector space of normalized MZVs of weight $n$, that is $\coeffof(\kontgraphs^n\tensor\qs) = \mzvswithhalf^n\tensor\qs$.
\end{theorem}
\begin{proof}
We shall write $\allstrings_\qs = \allstrings\tensor \qs$ and similarly for other $\ints$-modules in this section.

Let $\nautstrings = \stringof(\syntaxtrees) \subseteq \allstrings$.
By construction, $\nautstrings$ is closed under $\prepend,\append$ and $\shuffle$,
and contains the string $01 = \stringof(e)$ by \Cref{ex:hgraph_single_rung}.
We prove by induction that $\allstrings_\qs^n = \nautstrings_\qs^n$. We start at $n=2$, where $\allstrings_\qs^2$ is generated by $01 = \stringof(e)$.

Then suppose that $\allstrings_\qs^k = \nautstrings_\qs^k$ for all $k \leq n$.
We already have that $\nautstrings_\qs^{n+1} \subseteq \allstrings_\qs^{n+1}$ by construction, so for the reverse inclusion suppose we have some word $w \in \allstrings_\qs^{n+1}$.
By \Cref{prop:full_lyndon_decomposition}, $w$ can be
written as a $\qs$-linear combination of shuffles of Lyndon words of length $n+1$ or less in $\allstrings_\qs$:
\begin{align*}
    w = \sum_{i=1}^N a_i\prod^\shuffle_{1\leq k\leq n_i} l_{ik}
\end{align*}
By induction, all $l_{ik}$ of length $n$ or less are in $\nautstrings_\qs$.
By \Cref{lem:lyndons_startwith00} all $l_{ik}$ of length $n+1$
either start with 00 or end with 11.
If $l_{ik} = 00s$, then $0s \in \allstrings_\qs^n = \nautstrings_\qs^n$
so $l_{ik} = \prepend(0s) \in \nautstrings_\qs^{n+1}$.
Similarly, if $l_{ik} = s11$ then $l_{ik} = \append(s1) \in \nautstrings_\qs^{n+1}$.
So $w$ is expressible in terms of a $\qs$-linear shuffle product of words in $\nautstrings_\qs$, so $w \in \nautstrings_\qs$, and hence $w\in\nautstrings_\qs^{n+1}$ completing the induction.

Then using \Cref{prop:integration_procedure} and the fact that $L^{01}$ is surjective onto $\mzvs$, we have
\begin{align*}
    \mzvs_\qs^n = L^{01}(\allstrings^n_\qs) = L^{01}(\nautstrings^n_\qs) = L^{01}(\stringof(\syntaxtrees_\qs^n)) = \coeffof(\graphof(\syntaxtrees_\qs^n)) \subseteq \coeffof( \kontgraphs_\qs^n),
\end{align*}
so $\mzvs_\qs^n$ is generated by $\coeffof(\kontgraphs^n_\qs)$.

For the extended space $\mzvswithhalf_\qs^n = \mzvs_\qs^n + \frac{1}{2}\mzvs_\qs^{n-1}$,
we are left with generating $\frac{1}{2}\mzvs_\qs^{n-1}$ using graphs of weight $n$.
For each $z \in \mzvs_\qs^{n-1}$ the established inclusion $\mzvs_\qs^n \subseteq \coeffof( \kontgraphs_\qs^n)$ implies that there exists $\Gamma \in \kontgraphs^{n-1}_\qs$ such that $\coeffof(\Gamma) = z$.
Let $\Gamma_w$ be the single wedge graph (\Cref{ex:single_wedge}) which integrates to $-\frac{1}{2}$.
Then the join of the graphs $\Gamma*\Gamma_w$ is a graph of weight $n$ and we have $\coeffof( \Gamma*\Gamma_w) = -\frac{1}{2}z \in \frac{1}{2}\mzvs^{n-1}_\qs$.
It follows that $\mzvswithhalf_\qs^n = \coeffof(\kontgraphs_\qs^n)$.
\end{proof}

If we drop the weight restriction, we span all MZVs as a $\ints$-module, by observing that $\qs \subset \coeffof(\kontgraphs)$:

\begin{theorem}
    Integrals of Kontsevich graphs with the log propagator span the $\ints$-module of all MZVs, i.e., $\coeffof(\kontgraphs) = \mzvswithhalf$.
\end{theorem}
\begin{proof}
    Since $\coeffof( \kontgraphs) \subseteq \mzvswithhalf$ is a subring that generates $\mzvswithhalf$ over $\qs$ by \Cref{thm:span_weight_n_mzvs_over_q}, it suffices to show that $\qs \subset \coeffof(\kontgraphs)$.
    Then since $\{\frac{1}{p}\}_{p \text{ prime}}$ generates $\qs$ as a $\ints$-module, we need only show that $\frac{1}{p} \in \coeffof(\kontgraphs)$ for all primes $p$.
    For $p=2$, the single wedge graph (\Cref{ex:single_wedge}) integrates to $-\frac{1}{2}$.
    So let $p$ be an odd prime.
    We recall the formula for the Bernoulli numbers $B_{2k}$ for $k \geq 1$:
    \begin{align*}
        B_{2k} = -2(2k)! \frac{\zeta(2k)}{(2\pi i)^{2k}}.
    \end{align*}
    Let $\Gamma_k$ be the graph $\graphof(\prepend^{2k-2}(e)) \in \kontgraphs$,
    then $c_{\Gamma_k} = L^{01}(\stringof(\prepend^{2k-2}(e))) = L^{01}(0^{2k-1}1) = -\frac{\zeta(2k)}{(2\pi i)^{2k}}$,
    meaning $2(2k)!c_{\Gamma_k} = B_{2k}$,
    implying $B_{2k} \in \coeffof( \kontgraphs)$.

    Next suppose, inductively, that for all for primes $q < p$, we have $\frac{1}{q} \in \coeffof(\kontgraphs)$.
    By the Von Staudt--Clausen theorem, we have
    \begin{align*}
        B_{p-1} + \sum_{(q-1)|(p-1)}\frac{1}{q} &= m,
    \end{align*}
    where $q$ ranges over primes and $m \in \ints$.
    Note that $q=p$ is the largest prime occurring in the sum, so we can rewrite this equation as
    \begin{align*}
        \frac{1}{p} &= m - B_{p-1} - \sum_{\substack{(q-1)|(p-1)\\q < p}}\frac{1}{q}.
    \end{align*}
    We have $m \in \coeffof(\kontgraphs)$ (since $\frac{1}{2}\in\coeffof(\kontgraphs)$) and $B_{p-1}\in\coeffof(\kontgraphs)$ from above,
    so $\frac{1}{p}\in\coeffof(\kontgraphs)$.
\end{proof}

\subsection{Integer generation}

There are insufficiently many graphs in $\graphof(\syntaxtrees^n)$ to generate $\mzvs^n$ for all $n$ as a $\ints$-module, using known relations between MZVs.
We verified manually that $\coeffof( \graphof(\syntaxtrees^n)) = \mzvs^n$ holds for $n \leq 4$ but the case $n=5$ fails:
\begin{lem}
    Assuming the conjectured basis  $\{\zeta(2,3),\zeta(3,2)\}$ for $\mzvs^5\tensor\qs$ from \cite{hoffman97},
    we have $\coeffof(\graphof(\syntaxtrees^5)) \neq \mzvs^5$.
\end{lem}
\begin{proof}
We consider the following elements of $\syntaxtrees^5$ which form a generating set for $\stringof(\syntaxtrees^5)$, since $\stringof(\prepend(\append(\cdot))) = \stringof(\append(\prepend(\cdot)))$ and $*$ is commutative:
\begin{center}
\begin{tabular}{c|c|c}
    $t \in \syntaxtrees$ & $\stringof(t) \in \allstrings$ & $(2\pi i)^5 \coeffof( G(t))$ \\\hline
    $\prepend^3(e)$ & 00001                          & $-\zeta(5)$ \\
    $\prepend^2(\append(e))$ & 00011                 & $\zeta(1,4)$ \\
    $\prepend(e*e)$ & $2\cdot 00101 + 4 \cdot 00011$ & $2\zeta(2,3) + 4\zeta(1,4)$ \\
    $\prepend(e)*e$ & $01001 + 3\cdot 00101 + 6\cdot 00011$ & $\zeta(3,2) + 3\zeta(2,3) + 6\zeta(1,4)$ \\
    
    $\append^3(e)$ & 01111                           & $\zeta(1,1,1,2)$ \\
    $\prepend(\append^2(e))$ & 00111                 & $-\zeta(1,1,3)$ \\
    $\append(e*e)$  & $2\cdot 01011 + 4\cdot 00111$  & $-2\zeta(1,2,2) - 4\zeta(1,1,3)$ \\
    $\append(e)*e$  & $01101 + 3\cdot 01011 + 6\cdot 00111$ & $-\zeta(2,1,2) - 3\zeta(1,2,2) - 6\zeta(1,1,3)$ \\
\end{tabular}
\end{center}
We would like to know whether the third column spans $\mzvs^5$.
Note that MZV dualities (e.g., $\zeta(5) = \zeta(1,1,1,2)$ or $\zeta(1,2,2) = \zeta(2,3)$) makes the second half of the table redundant for this purpose.
We may write the first half of the table uniquely in the conjectured basis:
\begin{align*}
\begin{pmatrix}
    \zeta(5) \\
    \zeta(1,4) \\
    2\zeta(2,3) + 4\zeta(1,4) \\
    \zeta(3,2) + 3\zeta(2,3) + 6\zeta(1,4)
\end{pmatrix}
= 
\frac{1}{5}
\begin{pmatrix}
    4 & 6 \\
    -1 & 1 \\
    6 & 4 \\
    9 & 11 \\
\end{pmatrix}
\begin{pmatrix}
    \zeta(2,3) \\
    \zeta(3,2) \\
\end{pmatrix},
\end{align*}
and so after multiplying by 5, we see that we may only generate $\zeta(2,3)$ and $\zeta(3,2)$ over $\ints$ if the matrix on the right hand side has a left inverse over $\ints$.
However in $\mzvs^5 \tensor \intsmod{2}$, the matrix is
\begin{align*}
\begin{pmatrix}
    0 & 0 \\
    1 & 1 \\
    0 & 0 \\
    1 & 1 \\
\end{pmatrix}
\end{align*}
which has rank 1, so no left inverse exists. Hence $\mzvs^5$ is not generated as a $\ints$-module.
\end{proof}

\bibliographystyle{hyperamsalpha}
\bibliography{refs.bib} 

\end{document}